\documentclass[12pt, reqno, a4paper]{amsart}

\usepackage{ amssymb, amsmath, enumerate, amsfonts, amsthm, mathrsfs, url, bm, mathtools}

\usepackage{xcolor}  	
\usepackage{hyperref}
\hypersetup{
colorlinks,
   linkcolor={cyan!80!black},
   citecolor={cyan!80!black},
 urlcolor={cyan!80!black}
}

\usepackage{color}

\usepackage[margin=1in]{geometry}

\RequirePackage{doi}

\usepackage{amscd}
\usepackage{amsfonts}
\usepackage{float}
\usepackage{color}
\usepackage[
backend=biber,
style=alphabetic,
]{biblatex}
\usepackage{bookmark}

\renewbibmacro{in:}{}
\DeclareFieldFormat{title}{#1}

\DeclareFieldFormat[article]{title}{\mkbibemph{#1}}       
\DeclareFieldFormat[incollection]{title}{\mkbibemph{#1}}  
\DeclareFieldFormat[book]{title}{\mkbibemph{#1}}          
\DeclareFieldFormat[incollection]{booktitle}{#1}          
\DeclareFieldFormat[article]{journaltitle}{#1}            

\AtEveryBibitem{%
  \ifentrytype{misc}{\DeclareFieldFormat{title}{\mkbibemph{#1}}}{}}

\DeclareFieldFormat{eprint:eprint}{arXiv:\href{https://arxiv.org/abs/#1}{#1}}

\DeclareFieldFormat[inproceedings]{title}{\mkbibemph{#1}}

\DeclareFieldFormat[inproceedings]{booktitle}{#1}

\usepackage{amssymb}

\newtheorem{theorem}{Theorem}[section]
\newtheorem{lemma}{Lemma}[section]

\newtheorem{corollary}{Corollary}[section]
\newtheorem{proposition}{Proposition}[section]

\theoremstyle{definition}
\newtheorem{definition}{Definition}[section]

\theoremstyle{remark}

\numberwithin{equation}{section}

\newcommand{\Mod}[1]{\ (\mathrm{mod}\ #1)}

\renewcommand{\Re}{\mathrm{Re}}
\renewcommand{\Im}{\mathrm{Im}}

\newcommand{\R}{{\mathbb R}}

\newcommand{\Z}{{\mathbb Z}}
\newcommand{\E}{{\mathbb E}}

\renewcommand{\leq}{\leqslant}
\renewcommand{\geq}{\geqslant}

\renewcommand{\le}{\leqslant}
\renewcommand{\ge}{\geqslant}

\begin{document}

\title[Low moments of Hecke eigenvalue sums]{Low moments of Hecke eigenvalue sums}


\author{Jad Hamdan}
\address{Mathematical Institute, University of Oxford, Andrew Wiles Building\\ Radcliffe Observatory Quarter, Woodstock Rd\\
Oxford OX2 6GG\\
United Kingdom}
\email{hamdan@maths.ox.ac.uk}

\author{Sun-Kai Leung}
\address{Mathematical Institute, University of Oxford, Andrew Wiles Building\\ Radcliffe Observatory Quarter, Woodstock Rd\\
Oxford OX2 6GG\\
United Kingdom}
\email{sunkaileung@gmail.com}

\author{Mo Dick Wong}
\address{Department of Mathematics, Run Run Shaw Building, The University of Hong Kong, Pokfulam, Hong Kong}
\email{mdwong@hku.hk}


\date{}

\dedicatory{}


\begin{abstract}
We show that partial sums of the Sato--Tate random multiplicative functions introduced by Cogdell and Michel exhibit better-than-square-root cancellation. The proof proceeds via a connection to multiplicative chaos, following Harper's seminal work. By a non-trivial adaptation of Harper's derandomization argument for character sums, we also obtain upper bounds for low moments of Hecke eigenvalue sums and of Hecke eigenforms near the cusp $0$; to our knowledge, this is the first appearance of multiplicative chaos in the context of automorphic forms on $\mathrm{GL}(2)$. A novel ingredient is the introduction of Hecke $s$-norms.
\end{abstract}

\maketitle

\section{Introduction}
\textit{Random multiplicative functions} (RMFs) serve as models for classical multiplicative functions. For instance, the Steinhaus RMFs model Archimedean characters $n \mapsto n^{it}$ as $t$ varies over $[T,2T]$ for large $T$, or Dirichlet characters $\chi$ modulo a large prime (see \cite{MR1815216}, and \cite{harper2025bettersquarerootcancellationnumber} for an exposition). The extended Rademacher RMFs model real primitive characters $\chi_d$ as the fundamental discriminant $d$ varies over $[D,2D]$ for large $D$ (see \cite{MR2024414}, and \cite{delabretèche2026randommultiplicativefunctionsmaking} for the definition). To investigate the complex moments of symmetric power $L$-functions at $s=1$, Cogdell and Michel \cite{MR2035301} introduced the following RMFs, which we shall call the \textit{Sato--Tate random multiplicative functions}.

\begin{definition}[Sato--Tate random multiplicative function] \label{def:sato-tate}
Let $(\theta_p)_{p \text{ prime}}$ be a sequence of independent and identically distributed random variables, sampled from 
the Sato--Tate measure 
\begin{align*}
    \mu_{\mathrm{ST}}(d\theta) := \frac{2}{\pi} \sin^2\theta d\theta \qquad \text{for $\theta \in [0,\pi].$}
\end{align*}
We call $\mathbb{X}$ a \textit{Sato--Tate random multiplicative function} if, for each integer $v \geq 0,$ we define
\begin{align*}
\mathbb{X}(p^v):=
\sum_{j=0}^v e^{i(v-2j)\theta_p} 
\end{align*}
and extend the definition multiplicatively to any integer $n \geq 1$ by
\begin{align*}
\mathbb{X}(n):=\prod_{p} \mathbb{X}(p^{v_p(n)}),
\end{align*}
where $v_p(n)$ denotes the $p$-adic valuation of $n.$  
\end{definition}

In a breakthrough work, Harper \cite{MR4061962} determined the order of magnitude of low moments of partial sums of Steinhaus RMFs. The principal aim of this paper is to obtain the corresponding upper bound for Sato--Tate RMFs.\footnote{Concurrently, Max Wenqiang Xu and Junren Zheng obtained a lower bound of the same shape, which will also appear in the second author's upcoming work for a general class of RMFs. We thank them for generously sharing their results.} 

\begin{theorem} \label{thm:random}
Let $\mathbb{X}$ be a Sato--Tate random multiplicative function. Then uniformly for all sufficiently large $x$ and $0 \leq q \leq 1,$ we have
\begin{align*}
\mathbb{E} \left|\sum_{n \leq x}\mathbb{X}(n) \right|^{2q} \ll \left( \frac{x}{1+(1-q)\sqrt{\log \log x}} \right)^q.
\end{align*}
\end{theorem}

Sato--Tate RMFs model \textit{Hecke eigenvalues} $\lambda_f$ as $f$ varies over an orthogonal basis of \textit{normalized Hecke eigenforms} of large prime level (see \cite[p.~1570]{MR2035301} and \cite[p.~410]{MR3918448} for details).

Given integers $k \in 2\mathbb{N}$ and $N \in \mathbb{N}$, let $\mathcal{S}_k(\Gamma_0(N))$ denote the space of holomorphic cusp forms of weight $k$ and level $N.$ We also denote by $\mathcal{B}_k(N)$ an orthogonal basis of $\mathcal{S}_k(\Gamma_0(N)),$ and by $\mathcal{H}_k(N)$ such a basis consisting of normalized Hecke eigenforms (see \cite[Chapter 14]{MR2061214} for definitions). Every holomorphic cusp form $f \in \mathcal{S}_k(\Gamma_0(N))$ admits a Fourier expansion at the cusp $\infty$ given by
\begin{align*}
f(z)=\sum_{n \geq 1} \lambda_f(n)n^{\frac{k-1}{2}}e(nz) \qquad  \text{for $z \in \mathbb{H}$.}
\end{align*}
Here and throughout, we write $e(z):=e^{2\pi iz}.$
The coefficients $\lambda_f(1), \lambda_f(2),\ldots$ are called \textit{Hecke eigenvalues} if $f$ is a normalized Hecke eigenform.

Instead of varying uniformly over an orthogonal basis of holomorphic cusp forms, it is also natural to take the \textit{harmonic average}.

\begin{definition}[Harmonic average]
Given a function $\Phi : \mathcal{B}_k(N) \to \mathbb{C},$ we define its \textit{harmonic sum over the basis $\mathcal{B}_k(N)$} by
\begin{align*}
\mathbb{\sum}_{f \in \mathcal{B}_k(N)}^h \Phi(f) :=\frac{\Gamma(k-1)}{(4\pi)^{k-1}} \sum_{f \in \mathcal{B}_k(N)} \frac{\Phi(f)}{\|f\|_{\rm Pet}^2},
\end{align*}
where $\| \cdot\|_{\rm Pet}$ is the Petersson norm, and its \textit{normalized harmonic average over the basis $\mathcal{B}_k(N)$} by
\begin{align*}
\mathbb{E}_{f \in \mathcal{B}_k(N)}^h  \Phi(f):=\left( \mathbb{\sum}_{f \in \mathcal{B}_k(N)}^h 1 \right)^{-1} \mathbb{\sum}_{f \in \mathcal{B}_k(N)}^h \Phi(f).
\end{align*}
\end{definition}
Lamzouri \cite{MR3918448} initiated the study of partial sums of Hecke eigenvalues using Sato--Tate RMFs. Motivated by Harper's work \cite{harper2023typicalsizecharacterzeta} on low moments of character sums, we use Theorem \ref{thm:random} to establish an analogous upper bound for the low moments of partial sums of Hecke eigenvalues.\footnote{Concurrently, Max Wenqiang Xu and Junren Zheng obtained an analogous result in the weight aspect. We thank them for generously sharing their results.}
\begin{theorem} \label{thm:det}
Let $k \geq 2$ be an even integer and $N$ be a sufficiently large prime. Then uniformly for any $0 \leq q \leq 1$ and $1 \leq x \leq N_k,$ we have
\begin{align*}
\mathbb{E}_{f \in \mathcal{H}_k(N)}^h \left|\sum_{n \leq x}\lambda_f(n) \right|^{2q}  
\ll_{k} \left( \frac{x}{1+(1-q)\sqrt{\log \log (10L_k)}} \right)^q,
\end{align*}
where 
\begin{align*}
N_k:=
\begin{cases}
N/\log N & \mbox{{\normalfont if $k=2,$} } \\
\hfil N & \mbox{{\normalfont otherwise,} }
\end{cases}
\end{align*}
and $L_k:=\min\{ x,N_k/x \}.$
\end{theorem}

 As an immediate consequence, the partial sums of Hecke eigenvalues typically exhibit better-than-square-root cancellation.

\begin{corollary} \label{cor:markov} 
Let $k \geq 2$ be an even integer and $N$ be a sufficiently large prime. Then uniformly for any $t \geq 2$ and $1 \leq x \leq  N_k,$ we have
\begin{align*}
\mathbb{P}^h_{f \in \mathcal{H}_k(N)}\left( \left|\sum_{n \leq x}\lambda_f(n) \right| > t \cdot \frac{\sqrt{x}}{(\log \log (10L_k))^{1/4}} \right) \ll_{k}
\frac{\min \{ \log t, \sqrt{\log \log (10L_k)} \}}{t^2}.
\end{align*}
\end{corollary}

 Theorem \ref{thm:det} can also be used to bound the low moments of Hecke eigenforms near the cusp $0$, analogously to \cite[Corollary 2]{harper2023typicalsizecharacterzeta}. Moreover, we identify a phase transition at height $N^{-1}$, with multiplicative-chaos behavior persisting down to this scale and exponential cusp decay taking over below it.
\begin{theorem} \label{thm:modform}
Let $k \geq 2$ be an even integer and $N$ be a sufficiently large prime. Then uniformly for any $0 \leq q \leq 1$ and $r \geq 1,$ we have
\begin{align*} 
\mathbb{E}_{f \in \mathcal{H}_k(N)}^h |f(i/r)|^{2q} \ll_{k} 
\begin{cases}
\left( \dfrac{r^k}{1+(1-q)\sqrt{\log \log (10R)}} \right)^q
& \mbox{{\normalfont if $1 \leq r \leq N,$} } \\
\hfil \left( \left(\dfrac{r^{2}}{N}\right)^k \exp\left(- \dfrac{4\pi r}{N}\right) \right)^q
& \mbox{{\normalfont if $r > N$,} }
\end{cases}
\end{align*}
where  $R:=\min\{ r,N/r \}.$ 
\end{theorem}

The results in \cite{MR4061962} suggest that the bound in Theorem \ref{thm:random} reflects the true order of magnitude of the low moments of $\sum_{n\leq x}\mathbb{X}(n)$, and we similarly expect a matching lower bound in Theorem \ref{thm:det} to hold when $x\leq \sqrt{N}$. As for Theorem \ref{thm:modform}, recall that 
when $k \in \{2,4,6,8,10\}$, every Hecke eigenform of weight $k$ is primitive. It follows that if $f \in \mathcal{H}_k(N),$ then $f |_k W_N =\epsilon_f f$ for some $|\epsilon_f|=1$ (see (\ref{eq:fricke}) for the definition). In particular, for any $r>0,$ we have the symmetry
\begin{align*}
\left|f \left( \frac{i}{r} \right) \right| =  \left(\frac{r^2}{N} \right)^{k/2} 
\left|f\left(  \frac{ir}{N} \right) \right|,
\end{align*}
which suggests that a corresponding lower bound should also hold in Theorem \ref{thm:modform}.

\bigskip

\noindent\textit{Notation.} 
Throughout the paper, we use standard asymptotic notation, writing $f(T)=o(g(T))$ to mean that $|f(T)/g(T)|$ tends to $0$ as ${T\to\infty}$, and $f(T)=O(g(T))$ or $f(T)\ll g(T)$ to mean that $\limsup|f(T)/g(T)|$ is finite. The
implied constants depend only on the subscripted parameters unless
otherwise specified. We denote $e(x):=e^{2\pi i x}$ and $\exp_2(x):=\exp(e^x)$ for $x\in\mathbb R.$
For $n \in \mathbb{N}$, let $P^+(n)$ denote the largest prime factor of
$n$, with the convention $P^+(1)=1$. We say that $n$ is $y$-smooth if
$P^+(n)\leq y$, and write $\Psi(x,y):=\#\{n\leq x:P^+(n)\leq y\}.$ If $\mathcal P$ is a set of primes, we write
$\omega_{\mathcal P}(n):=\#\{p\in\mathcal P:p\mid n\}.$\\

\section*{Acknowledgements}
The authors would like to thank Adam Harper and James Maynard for helpful discussions, as well as Max Wenqiang Xu and Junren Zheng for generously sharing their results. S.-K. L. is supported by the
Croucher Fellowship for Postdoctoral Research.
M.D.W. is supported by a start-up fund from the Faculty of Science and the URC Seed Fund for Basic Research at The University of Hong Kong, and by a start-up allowance from the Croucher Foundation. J.H.\ is supported by the Engineering and Physical Sciences Research Council (EPSRC) Grant EP/Z535990/1.

\section{Preliminaries}

This section collates several technical lemmas used throughout the paper. We begin with the following partition of unity, which is used in Section \ref{sec:derandom} to (roughly) fix the value of truncated Euler product integrals.
\begin{lemma}[Partition of unity] \label{lem:partition}
Let $J \in \mathbb{N}$ be large, and $\delta > 0$ be small. Then there exist functions $g : \R \rightarrow \R$ (depending on $\delta$) and $g_{J+1} : \R \rightarrow \R$ (depending on $\delta$ and $J$) such that, if we define $g_{j}(x) := g(x - j)$ for $|j| \leq J$, we have the following properties:
\begin{itemize}

\item $\sum_{|j| \leq J} g_{j}(x) + g_{J+1}(x) = 1$ for $x \in \R;$

\item $g(x) \geq 0$ for $x \in \R$, and $g(x) \leq \delta$ whenever $|x| > 1;$

\item $g_{J+1}(x) \geq 0$ for $x \in \R$, and $g_{J+1}(x) \leq \delta$ whenever $|x| \leq J;$

\item we have the derivative estimate $|\frac{d^{l}}{dx^{l}} g(x)| \leq \frac{1}{\pi (l+1)} (\frac{2\pi}{\delta})^{l+1}$ for $l \in \mathbb{N}$ and $x \in \R.$
\end{itemize}

\begin{proof}
See \cite[Approximation result 1]{harper2023typicalsizecharacterzeta}.
\end{proof}

\end{lemma}

\noindent Sato--Tate RMFs, as Steinhaus RMFs, satisfy an orthogonality relation.

\begin{lemma}[Orthogonality relation] \label{lem:orthog}
Let $m,n \geq 1$ be integers. Then
\begin{align*}
\mathbb{E} [\mathbb{X}(m) \mathbb{X}(n)]=\delta_{mn}.
\end{align*}

\begin{proof}
By the multiplicativity and independence of $\mathbb{X}$ (see Definition \ref{def:sato-tate}), it suffices to show that for each prime $p \geq 2$ and any integers $u,v \geq 0,$ we have
\begin{align*}
\mathbb{E}[\mathbb{X}(p^{u}) \mathbb{X}(p^{v})] = \delta_{uv}.
\end{align*}
This can be verified by elementary algebra and calculus, and the lemma follows.
\end{proof}

\end{lemma}

\noindent Whilst Sato--Tate RMFs are not completely multiplicative, they satisfy the Hecke recursion, as Hecke eigenvalues do.

\begin{lemma}[Hecke recursion] \label{lem:recursion}
Let $\mathbb{X}$ be a Sato--Tate random multiplicative function. Then for any integers $m,n \geq 1,$ we have
\begin{align}
\mathbb{X}(m) \mathbb{X}(n) 
= \sum_{\substack{d \mid (m,n)}} \mathbb{X} \left( \frac{mn}{d^2}\right). \label{eq:randomhecke}
\end{align}
Given an even integer $k \geq 2$ and a prime $N \geq 2,$ let $f\in\mathcal H_k(N).$ Then for any integers $1 \leq m,n<N,$ we have
\begin{align}
\lambda_f(m) \lambda_f(n)
= \sum_{\substack{d \mid (m,n)}} \lambda_f \!\left( \frac{mn}{d^2}\right).
\label{eq:dethecke}
\end{align}

\begin{proof}
To establish (\ref{eq:randomhecke}), by the multiplicativity of $\mathbb{X}$ (see Definition \ref{def:sato-tate}), it suffices to show that for each prime $p \geq 2$ and any integers $u,v \geq 0,$ we have
\begin{align*}
\mathbb{X}(p^{u}) \mathbb{X}(p^{v}) = \sum_{j=0}^{\min( u,v)} \mathbb{X}(p^{u+v-2j}),
\end{align*}
which can be verified using elementary algebra and calculus.

To establish (\ref{eq:dethecke}), we recall the Hecke recursion 
\begin{align}
\lambda_f(r) \lambda_f(s)
= \sum_{\substack{d \mid (r,s)}} \lambda_f \!
\left( \frac{rs}{d^2}\right) \label{eq:coprimehecke}
\end{align}
for integers $r,s \geq 1$ satisfying $(rs,N) = 1$ (see, e.g., \cite[Section 14.7]{MR2061214}).\footnote{One must be wary of the normalization. Also, due to the additional coprimality condition, the assumption that the Hecke eigenforms are primitive is unnecessary for \eqref{eq:coprimehecke} to hold.} Since by assumption $m,n<N$ and $N$ is prime, the coprimality condition holds and (\ref{eq:coprimehecke}) becomes (\ref{eq:dethecke}). Therefore, the lemma follows.
\end{proof}

\end{lemma}

Unlike Dirichlet characters of prime modulus, Hecke eigenvalues do not satisfy an exact orthogonality relation. Nevertheless, the deviation from exact orthogonality is explicit.

\begin{lemma}[Petersson trace formula] \label{lem:peter}
Let $m,n \geq 1$ be integers. Then
\begin{align*}
\mathbb{\sum}_{f \in \mathcal{B}_k(N)}^h \lambda_f(m) \overline{\lambda_f(n)}=\delta_{mn}+ 2\pi i^{-k}
\sum_{\substack{c \geq 1\\c \equiv 0 \Mod{N}}}
\frac{S(m,n;c)}{c} \cdot J_{k-1}\left( \frac{4\pi \sqrt{mn}}{c} \right).
\end{align*}

\begin{proof}
See, e.g., \cite[Corollary 14.23]{MR2061214}.
\end{proof}

\end{lemma}

\noindent As a consequence of the trace formula, we establish an asymptotic bilinear estimate which may be of independent interest (see \cite[Section 8]{MR2330439} for similar results).

\begin{lemma}[Asymptotic bilinear estimate] \label{lem:largesieve}
Let $k \geq 2$ be an even integer and let $N$ be a sufficiently large positive integer. Then for any $1 \leq x \leq N$ and sequences of complex numbers $\boldsymbol{\alpha}=(\alpha_m)_{m \leq x}, \boldsymbol{\beta}=(\beta_n)_{n \leq x},$ we have
\begin{gather*}
\left|\mathbb{E}^h_{f \in \mathcal{B}_k(N)}  \left( \sum_{m \leq x} {\alpha}_m\lambda_f(m) \right)\overline{\left( \sum_{n \leq x} {\beta}_n\lambda_f(n) \right)} 
-\sum_{m \leq x} {\alpha}_m \overline{{\beta}_m} \right| \\
\ll_{k} 
\begin{cases}
\hfil \dfrac{x\log (2x)}{N}  \|\boldsymbol{\alpha}\|_2 \|\boldsymbol{\beta}\|_2 & \mbox{{\normalfont if $k=2,$ }} \\
\left(\dfrac{x}{N} \right)^{k-1}  \|\boldsymbol{\alpha}\|_2 \|\boldsymbol{\beta}\|_2 
 & \mbox{{\normalfont otherwise, }}
\end{cases}
\end{gather*}
where
\begin{align*}
\|\boldsymbol{\alpha}\|_2:=\left(\sum_{m \leq x} |{\alpha}_m|^2\right)^{1/2}, 
\|\boldsymbol{\beta}\|_2:=\left(\sum_{n \leq x} |{\beta}_n|^2\right)^{1/2}.
\end{align*}
In particular, we have
\begin{align*}
\mathbb{E}^h_{f \in \mathcal{B}_k(N)}  \left| \sum_{m \leq x} {\alpha}_m\lambda_f(m) \right|^2
=\left(1+O_k \left( \frac{x}{N} \cdot  (\log (2x))^{\max\{ 0,3-k \}} \right) \right) \|\boldsymbol{\alpha}\|_2^2.
\end{align*}

\begin{proof}
For $k \geq 4,$ we follow the proof of \cite[Theorem 5.7]{MR1474964} in the case $1 \leq x \leq N.$ Applying Lemma 
\ref{lem:peter}, we have
\begin{gather}
\mathbb{\sum}^h_{f \in \mathcal{B}_k(N)}  \left( \sum_{m \leq x} {\alpha}_m\lambda_f(m) \right)\overline{\left( \sum_{n \leq x} {\beta}_n\lambda_f(n) \right)} \nonumber \\
=\sum_{m \leq x} \sum_{n \leq x} \alpha_m \overline{\beta_n} 
\left( 
\delta_{mn}+ 2\pi i^{-k}
\sum_{\substack{c \geq 1\\c \equiv 0 \Mod{N}}}
\frac{S(m,n;c)}{c} \cdot J_{k-1}\left( \frac{4\pi \sqrt{mn}}{c} \right)
\right) \nonumber \\
=\sum_{m \leq x} \alpha_m \overline{\beta_m} 
+2\pi i^{-k} \sum_{\substack{c \geq 1\\c \equiv 0 \Mod{N}}} \frac{1}{c}
\sum_{m \leq x} \sum_{n \leq x} \alpha_m \overline{\beta_n} S(m,n;c) J_{k-1}\left( \frac{4\pi \sqrt{mn}}{c} \right). \label{eq:afterpeter}
\end{gather}
Expanding the Bessel function in its power series
\begin{align*}
J_{k-1}(z)=\sum_{\ell=0}^{\infty} \frac{(-1)^{\ell}}{\ell! \Gamma(k+\ell)} \left(\frac{z}{2} \right)^{k-1+2\ell},
\end{align*}
we obtain, by absolute convergence, that
\begin{gather}
\sum_{m \leq x} \sum_{n \leq x} \alpha_m \overline{\beta_n} S(m,n;c) J_{k-1}\left( \frac{4\pi \sqrt{mn}}{c} \right) \nonumber \\
=\sum_{\ell=0}^{\infty} \frac{(-1)^{\ell}}{\ell! \Gamma(k+\ell)} \left( \frac{2\pi}{c} \right)^{k-1+2\ell}
\sum_{m \leq x} \sum_{n \leq x} \alpha_m \overline{\beta_n} (mn)^{\nu_{\ell}} S(m,n;c), \label{eq:afterbessel}
\end{gather}
where $\nu_{\ell}:=\frac{k-1}{2}+\ell.$ Opening the Kloosterman sum gives
\begin{align*}
\sum_{m \leq x} \sum_{n \leq x} \alpha_m \overline{\beta_n} (mn)^{\nu_{\ell}} S(m,n;c)
=\sum_{\substack{d \Mod{c}\\(d,c)=1}} 
\left( \sum_{m \leq x} \alpha_m m^{\nu_{\ell}}  e\left( \frac{md}{c} \right) \right)
\left( \sum_{n \leq x} \overline{\beta_n} n^{\nu_{\ell}}  e\left( \frac{n\overline{d}}{c} \right) \right).
\end{align*}
By the Cauchy--Schwarz inequality, this is
\begin{align*}
\leq \left(\sum_{\substack{d \Mod{c}\\(d,c)=1}}  
\left|  \sum_{m \leq x} \alpha_m m^{\nu_{\ell}}  e\left( \frac{md}{c} \right) \right|^2 
\right)^{1/2}
\left(\sum_{\substack{d \Mod{c}\\(d,c)=1}}  
\left|  \sum_{n \leq x} \overline{\beta_n} n^{\nu_{\ell}}  e\left( \frac{n\overline{d}}{c} \right) \right|^2 
\right)^{1/2}.
\end{align*}
Since $x \leq N \leq c,$ it follows from additive orthogonality that 
\begin{align*}
\sum_{\substack{d \Mod{c}\\(d,c)=1}}  
\left|  \sum_{m \leq x} \alpha_m m^{\nu_{\ell}}  e\left( \frac{md}{c} \right) \right|^2 
\leq& 
\sum_{\substack{d \Mod{c}}}  
\left|  \sum_{m \leq x} \alpha_m m^{\nu_{\ell}}  e\left( \frac{md}{c} \right) \right|^2\\
=&c \sum_{m \leq x} |\alpha_m|^2 m^{2\nu_{\ell}} \\
\leq & c x^{k-1+2\ell} \|\boldsymbol{\alpha}\|_2^2,
\end{align*}
and therefore
\begin{align*} 
\left| \sum_{m \leq x} \sum_{n \leq x} \alpha_m \overline{\beta_n} (mn)^{\nu_{\ell}} S(m,n;c) \right|
\leq c  x^{k-1+2\ell} \|\boldsymbol{\alpha}\|_2 \|\boldsymbol{\beta}\|_2.
\end{align*}
Combining with (\ref{eq:afterpeter}) and (\ref{eq:afterbessel}), we have
\begin{gather*}
\left| \mathbb{\sum}^h_{f \in \mathcal{B}_k(N)}  \left( \sum_{m \leq x} {\alpha}_m\lambda_f(m) \right)\overline{\left( \sum_{n \leq x} {\beta}_n\lambda_f(n) \right)} -  \sum_{m \leq x} \alpha_m \overline{\beta_m} 
\right| \\
\leq 2\pi \|\boldsymbol{\alpha}\|_2 \|\boldsymbol{\beta}\|_2 
\sum_{\ell=0}^{\infty} \frac{(2\pi x)^{k-1+2\ell}}{\ell! \Gamma(k+\ell)} \sum_{\substack{c \geq 1\\c \equiv 0 \Mod{N}}} \frac{1}{c^{k-1+2\ell}} \\
= 2\pi \|\boldsymbol{\alpha}\|_2 \|\boldsymbol{\beta}\|_2  
\sum_{\ell=0}^{\infty} \frac{(2\pi)^{k-1+2\ell}\zeta(k-1+2\ell)}{\ell! \Gamma(k+\ell)}
\left(\frac{x}{N} \right)^{k-1+2\ell}.
\end{gather*}
Since $\zeta(k-1+2\ell)<\infty$ for $k \geq 4$ and $\ell \geq 0,$ this is
\begin{align} \label{eq:withoutnormalize}
\ll_k \left(\frac{x}{N} \right)^{k-1}  \|\boldsymbol{\alpha}\|_2 \|\boldsymbol{\beta}\|_2. 
\end{align}
In particular, taking $x=1, \alpha_{1}=1, \beta_1=1$ gives
\begin{align} \label{eq:n=1general}
\mathbb{\sum}_{f \in \mathcal{B}_k(N)}^h 1 = 1+O_k(N^{1-k}),
\end{align}
and the lemma for $k \geq 4$ follows from (\ref{eq:withoutnormalize}).

For $k=2,$ appealing to \cite[Theorem 1]{MR1258904}, we have
\begin{align} \label{eq:dfi2}
\mathbb{\sum}^h_{f \in \mathcal{B}_2(N)}  \left| \sum_{m \leq x} {\alpha}_m\lambda_f(m) \right|^2
=\left(1+O \left(  \frac{x\log (2x)}{N}  \right) \right) \|\boldsymbol{\alpha}\|_2^2.
\end{align}
In particular, taking $x=1, \alpha_{1}=1$ gives
\begin{align} \label{eq:n=1k=2}
\mathbb{\sum}_{f \in \mathcal{B}_2(N)}^h 1 = 1+O(1/N),
\end{align}
and the lemma for $k = 2$ follows from (\ref{eq:dfi2}) by standard polarization.
\end{proof}

\end{lemma}

We also extend the asymptotic bilinear estimate beyond the level $N$ when $\boldsymbol{\alpha}=\boldsymbol{\beta}$ is dyadically supported (see \cite[Theorem 1]{MR1258904} without the dyadic support assumption, but with an additional logarithmic loss in weight $2$).

\begin{lemma}[Large sieve inequality] \label{lem:largesieveineq2}
Let $k \geq 2$ be an even integer and $N \geq 1$ be an integer. Then for any real number $x \geq 1$ and a sequence of complex numbers $\boldsymbol{\alpha}=(\alpha_m)_{x<m \leq 2x},$ we have 
\begin{align*}
\mathbb{E}^h_{f \in \mathcal{B}_k(N)}  \left| \sum_{x<m \leq 2x} {\alpha}_m\lambda_f(m) \right|^2
\ll_k \left(1+\frac{x}{N} \right)\|\boldsymbol{\alpha}\|_2^2,
\end{align*}
where
\begin{align*}
\|\boldsymbol{\alpha}\|_2:=
\left(\sum_{x<m \leq 2x} |{\alpha}_m|^2\right)^{1/2}.
\end{align*}

\begin{proof}
For $k \geq 4,$ this is an immediate consequence of \cite[Theorem 5.7]{MR1474964} with (\ref{eq:n=1general}). For $k=2,$ we first show that if $F \in \mathcal{S}_2(\Gamma_0(N)),$ then
\begin{align} \label{eq:dyadicpet}
\sum_{M < m \leq 2M} |\lambda_F(m)|^2 \ll \left( 1+\frac{M}{N} \right) \|F\|_{\rm Pet}^2
\end{align}
for any $M \geq 1$, where 
\begin{align*}
\|F\|_{\rm Pet}:=\left( \int_{\Gamma_0(N) \backslash \mathbb{H}} |F(x+iy)|^2 dxdy \right)^{1/2}
\end{align*}
is the Petersson norm of $F.$

Let $a_F(n):=\lambda_F(n)\sqrt{n}$ for integers $n \geq 1.$ By Parseval's identity in the $x$-variable, one can show that for any $y>0,$ we have
\begin{align*}
\int_0^1 |F(x+iy)|^2 dx = \sum_{m \geq 1} |a_F(m)|^2 e^{-4 \pi m y },
\end{align*}
so that
\begin{align*}
\sum_{M < m \leq 2M} |\lambda_F(m)|^2 \ll  \int_{M^{-1}}^{\infty} \int_0^1 |F(x+iy)|^2 dx dy.
\end{align*}
Arguing analogously to \cite[p. 103]{MR1317646}, this is
\begin{gather*}
\leq \max_{z \in P_M} \# \{ \gamma \in \Gamma_0(N) /\{\pm1\} \,:\, \gamma z \in P_M  \} \cdot \| F\|_{\rm Pet}^2 \\
\leq \left( 1+\frac{3M}{N} \right)\|F\|_{\rm Pet}^2,
\end{gather*}
where $P_M:=\{ z \in \mathbb{H}: 0<\Re(z) \leq 1, \Im(z) > M^{-1} \},$ and the bound (\ref{eq:dyadicpet}) is verified.

Returning to the proof of the large sieve inequality, opening the square and  rearranging gives
\begin{align*}
\mathbb{\sum}^h_{f \in \mathcal{B}_2(N)}  \left| \sum_{x<m \leq 2x} {\alpha}_m\lambda_f(m) \right|^2
=\sum_{x<m \leq 2x} \alpha_m \left(\mathbb{\sum}^h_{f \in \mathcal{B}_2(N)}
\lambda_f(m) \sum_{x<n \leq 2x} \overline{\alpha_n \lambda_f(n)}\right).
\end{align*}
By the Cauchy--Schwarz inequality, this is
\begin{align*}
\leq \|\boldsymbol{\alpha}\|_2 \left( 
\sum_{x<m \leq 2x} 
\left|\mathbb{\sum}^h_{f \in \mathcal{B}_2(N)}
\lambda_f(m) \sum_{x<n \leq 2x} \overline{\alpha_n \lambda_f(n)}\right|^2
\right)^{1/2}.
\end{align*}
Set
\begin{align*}
A_f:=\sum_{x<n \leq 2x} \overline{\alpha_n \lambda_f(n)}
\end{align*}
and
\begin{align*}
F:=\frac{1}{4\pi}\sum_{f \in \mathcal{B}_2(N)}\frac{A_f}{\|f\|_{\rm Pet}^2}f.
\end{align*}
Then $\lambda_F(m)=\mathbb{\sum}_{f \in \mathcal{B}_2(N)}^h\lambda_f(m)A_f.$ Applying (\ref{eq:dyadicpet}), we obtain
\begin{align} \label{eq:beforeorthog}
\mathbb{\sum}^h_{f \in \mathcal{B}_2(N)}  \left| \sum_{x<m \leq 2x} {\alpha}_m\lambda_f(m) \right|^2
\ll \left( 1+\frac{x}{N}\right)^{1/2} \|\boldsymbol{\alpha}\|_2  \| F \|_{\rm Pet}.
\end{align}
By the orthogonality of $ \mathcal{B}_2(N),$ we have
\begin{align*}
\| F \|_{\rm Pet}^2
&=\frac{1}{16\pi^2}\sum_{f \in \mathcal{B}_2(N)}\frac{|A_f|^2}{\|f\|_{\rm Pet}^2}\\
&=\frac{1}{4\pi}\mathbb{\sum}^h_{f \in \mathcal{B}_2(N)}|A_f|^2.
\end{align*}
Combining with (\ref{eq:beforeorthog}), we obtain
\begin{align*}
\mathbb{\sum}^h_{f \in \mathcal{B}_2(N)}  \left| \sum_{x<m \leq 2x} {\alpha}_m\lambda_f(m) \right|^2
\ll \left( 1+\frac{x}{N} \right) \|\boldsymbol{\alpha}\|_2^2.
\end{align*}
Finally, combining with (\ref{eq:n=1k=2}), the lemma for $k=2$ follows.
\end{proof}

\end{lemma}

\section{Proof of Theorem \ref{thm:random}}

To prove Theorem \ref{thm:random}, we first reduce the problem to that of bounding low moments of an integral involving a truncated random Euler product. For $n\geq 2$, let $P^+(n)$ denote the largest prime factor of $n$, and set $P^+(1):=1$. Given a real number $P \geq 2,$ we define the random Dirichlet series
\begin{align*} 
F_{P}^{\mathrm{rand}}(s):=\sum_{\substack{n \geq 1\\ P^+(n) \leq P}} \frac{\mathbb{X}(n)}{n^s}
\end{align*}
for $\Re(s)>0.$ It can be expressed as a truncated random Euler product of degree $2$ as follows.

\begin{lemma} \label{lem:euler}
Let $P \geq 2.$ Then for $\Re(s)>0,$ we have
\begin{align*}
F_{P}^{\mathrm{rand}}(s)=\prod_{p \leq P} \left( 1-\frac{e^{i\theta_p}}{p^{s}}\right)^{-1} \left( 1-\frac{e^{-i\theta_p}}{p^{s}} \right)^{-1}.
\end{align*}

\begin{proof}
This is an immediate consequence of Definition \ref{def:sato-tate}.
\end{proof}

\end{lemma}

\begin{lemma}\label{lem:random-reduce-to-mass}
Uniformly for $y \in [2, \sqrt{x}]$ and $q \in [0,1]$,
there exist absolute constants $C \geq 0$ and $c > 0$ such that
\begin{gather}
\mathbb{E}\left[ \left|\frac{1}{\sqrt{x}}\sum_{n \leq x}\mathbb{X}(n) \right|^{2q}\right] \ll \mathbb{E}\left[\left(\frac{1}{\log y} \int_{\mathbb{R}} \left| \frac{F_{y}^{\mathrm{rand}}(\frac{1}{2} + it)}{\frac{1}{2} + it}\right|^2 dt\right)^q \right] + \left[(\log y)^C e^{-c \frac{\log x}{ \log y}}\right]^q.
\end{gather}

\begin{proof}
One observes that the argument used in the proof of \cite[Lemma 1.2]{GW2025} carries over verbatim to the Sato--Tate random multiplicative functions, as it relies solely on the orthogonality relation in Lemma \ref{lem:orthog}, with uniformity over all $q \in [0,1]$. The only difference in the present setting is that the Dirichlet series restricted to $y$-smooth numbers is given by the truncated Euler product in Lemma \ref{lem:euler}.
\end{proof}

\end{lemma}

It therefore remains to establish the following analogue of \cite[Lemma 1.3]{GW2025} for Sato--Tate random multiplicative functions, whose proof is given across the next three sections.

\begin{proposition} \label{prop:random-mass-bound}
Uniformly in $y \geq 3$ and $q \in [0,1]$, we have
\begin{align*}
\mathbb{E}\left[\left(\frac{1}{\log y} \int_{\mathbb{R}} \bigg| \frac{F_{y}^{\mathrm{rand}}(\frac{1}{2} + it)}{\frac{1}{2} + it}\bigg|^2 dt\right)^q \right] 
\ll \big( 1+(1-q)\sqrt{\log \log y} \big)^{-q}.
\end{align*}

\end{proposition}

\noindent Indeed, Theorem \ref{thm:random} follows upon combining Lemma \ref{lem:random-reduce-to-mass} 
with Proposition \ref{prop:random-mass-bound}, taking $\log y := \log x / (\log \log x)^2$.

\section{Estimates for truncated Euler products}

The proof of Proposition \ref{prop:random-mass-bound} requires several estimates for moments of truncated random Euler products.

\begin{lemma} \label{lem:prime-sum}
Let $t \in \mathbb{R}.$ Then for any sufficiently large $P,$ we have
\begin{align}
\label{eq:longprimesum}
\sum_{p \leq P} \frac{1}{p^{1+it}} = \log \left( 1+\min \left\{ \log P, \frac{1}{|t|}\right\} \right) + O(\log \log (3+|t|)).
\end{align}
At $t=0$, the minimum on the right-hand side is interpreted as $\log P$.
Moreover, there exists some $c > 0$ such that the estimates
\begin{gather}
\label{eq:psum-needed-random1}
\sum_{X \le \log p \le eX} \frac{1}{p^{1+i\tau}} 
= \int_X^{eX} \frac{e^{-i\tau v}}{v}dv + O\left((1+|\tau|) X e^{-c\sqrt{X}}\right), \\
\label{eq:psum-needed-random2}
 \qquad \left|\sum_{X \le \log p \le eX} \frac{1}{p^{1+i\tau}}\right| 
 \ll \min\left(1, \frac{1}{|\tau|X}\right) + (1+|\tau|) X e^{-c\sqrt{X}}
\end{gather}

\noindent hold uniformly for all $X \ge 3$ and $\tau \in \mathbb{R}$.

\begin{proof}
We first show that
\begin{align}
\sum_{p \leq P} \frac{1}{p^{1+it}}=\int_{2}^P \frac{1}{x^{1+it}} \cdot\frac{dx}{\log x} + O(\log \log (3+|t|)). \label{eq:stot}
\end{align}
Recall de la Vall\'ee Poussin's prime number theorem, which states that $\pi(x)=\operatorname{Li}(x)+E(x)$ for $x \geq 2,$ where $E(x) \ll x \exp(-c\sqrt{\log x})$ for some absolute constant $c>0.$ Let $T:=3+|t|$ and $Y:=\min\{ P, \exp(10c^{-2}(\log T)^2) \}.$ Then by partial summation, we have
\begin{gather}
\sum_{Y<p \leq P} \frac{1}{p^{1+it}} = \int_{Y}^P \frac{1}{x^{1+it}} \cdot\frac{dx}{\log x} + \frac{E(P)}{P^{1+it}} - \frac{E(Y)}{Y^{1+it}}+(1+it) \int_Y^P \frac{E(x)}{x^{2+it}} dx \nonumber\\
=\int_{Y}^P \frac{1}{x^{1+it}} \cdot\frac{dx}{\log x} +
O((1+|t|)(1+\sqrt{\log Y})\exp(-c\sqrt{\log Y})). \label{eq:yp}
\end{gather}
On the other hand, Mertens' estimate gives
\begin{align}
\left|\sum_{p \leq Y} \frac{1}{p^{1+it}} - \int_{2}^Y \frac{1}{x^{1+it}} \cdot\frac{dx}{\log x} \right| 
\leq& \sum_{p \leq Y} \frac{1}{p} + \int_2^Y \frac{dx}{x\log x} \nonumber\\
\ll& \log \log Y. \label{eq:2y}
\end{align}
Combining (\ref{eq:yp}) and (\ref{eq:2y}), the estimate (\ref{eq:stot}) follows with our choice of $Y$. It remains to show 
\begin{align}
\int_{2}^P \frac{1}{x^{1+it}} \cdot\frac{dx}{\log x} = 
\log \left( 1+\min \left\{ \log P, \frac{1}{|t|}\right\} \right)
+O(1). 
\label{eq:logintegral}
\end{align}
Making the change of variables $u=\log x,$ the integral on the left-hand side becomes
\begin{align*}
\int_{\log 2}^{\log P} e^{-itu} \cdot\frac{du}{u}.
\end{align*}
Suppose $|t| \leq 1/\log P.$ Then $e^{-itu}=1+O(|t|u),$ so that
\begin{align}
\int_{\log 2}^{\log P} e^{-itu} \cdot\frac{du}{u}= \log \log P +O(1).
\label{eq:tsmall}
\end{align}
Otherwise, suppose $|t| > 1/\log P.$ Then similarly, we have
\begin{align} \label{eq:tbigg}
\int_{\log 2}^{\log P} e^{-itu} \cdot\frac{du}{u}=&
\left\{\int_{\log 2}^{\max\{ \log 2, |t|^{-1} \}}+ \int_{\max\{ \log 2, |t|^{-1} \}}^{\log P} \right\}
e^{-itu} \cdot \frac{du}{u} \nonumber \\
=& \log (\max \{\log 2, |t|^{-1}\})+O(1).
\end{align}
Combining (\ref{eq:tsmall}) and (\ref{eq:tbigg}), the estimate (\ref{eq:logintegral}) follows.
The estimates \eqref{eq:psum-needed-random1} \eqref{eq:psum-needed-random2} follow from a similar analysis based on \eqref{eq:yp}, and the proof of the lemma is complete.
\end{proof}

\end{lemma}

\begin{lemma} \label{lem:beta}
Let $\beta \geq 1$ and $t \in \mathbb{R}.$ Then for any sufficiently large $P$, we have
\begin{align} 
\mathbb{E} \left|F^{\mathrm{rand}}_{P}\left(\frac{1}{2}+it\right)\right|^{2\beta}
\asymp_{\beta} (\log P)^{\beta^2} \exp\left(\beta(\beta-1)\sum_{p \le P} \frac{\cos(2t \log p) }{p}  \right). \label{eq:betafirst}
\end{align}
Moreover, we have
\begin{gather}
\mathbb{E} \left|F^{\mathrm{rand}}_{P}\left(\frac{1}{2}+it\right)\right|^{2\beta}
\nonumber \\
\ll_{\beta} (\log P)^{\beta^2} 
\left( 
1+\min \left\{  \log P, \frac{1}{|t|} \right\}
\right)^{\beta(\beta-1)} (\log (3+|t|))^{O(\beta(\beta-1))}. \label{eq:betasecond}
\end{gather}

\begin{proof}
By the Taylor expansion, we have
\begin{align} \label{eq:taylor-logep}
    \log  \left|F^{\mathrm{rand}}_{P}\left(\frac{1}{2}+it\right)\right|
    & = \mathrm{Re}\sum_{p \le P}\sum_{k \ge 1} \frac{1}{k} 
    \left( \left(\frac{e^{i\theta_p}}{p^{\frac{1}{2} + it}}\right)^k + \left(\frac{e^{-i\theta_p}}{p^{\frac{1}{2} + it}}\right)^k \right) \nonumber\\
    & = \sum_{p \le P} \sum_{k\ge 1} \frac{2\cos(k\theta_p) \cos(kt \log p)}{kp^{k/2}}. 
\end{align}

It is straightforward to check that
\begin{align} \label{eq:ST-moments}
    \mathbb{E}[\cos(\theta_p)] = 0,
    \quad  \mathbb{E}[\cos^2(\theta_p)] = \frac{1}{4}
    \quad \text{and} \quad \mathbb{E}[\cos(2\theta_p)] = \mathbb{E}[2\cos^2(\theta_p) - 1] = -\frac{1}{2}.
\end{align}

For each prime $p$, let
\begin{align*}
L_p(t):=\mathbb{E}\exp\left(4\beta\sum_{k\geq 1}
\frac{\cos(k\theta_p)\cos(kt\log p)}{kp^{k/2}}\right).
\end{align*}
By independence, we have
\begin{align*}
\mathbb{E} \left|F^{\mathrm{rand}}_{P}\left(\frac{1}{2}+it\right)\right|^{2\beta}
=\prod_{p\leq P}L_p(t).
\end{align*}
For all sufficiently large primes $p$, Taylor's theorem and the preceding identities give, uniformly in $t$, that
\begin{align*}
L_p(t)=1+\frac{2\beta^2\cos^2(t\log p)-\beta\cos(2t\log p)}{p}
+O_{\beta}(p^{-3/2}).
\end{align*}
The finitely many remaining local factors are bounded above and below by positive constants depending only on $\beta$. Taking logarithms and using $2\cos^2 u=1+\cos(2u)$, we obtain
\begin{align*}
\mathbb{E} \left|F^{\mathrm{rand}}_{P}\left(\frac{1}{2}+it\right)\right|^{2\beta}
\asymp_{\beta}
\exp\left(\beta^2 \sum_{p \leq P} \frac{1}{p} + \beta(\beta-1)\sum_{p \leq P} \frac{\cos(2t \log p) }{p}  \right).
\end{align*}
The estimate (\ref{eq:betafirst}) now follows from Mertens' estimate. Finally, the bound (\ref{eq:betasecond}) follows from (\ref{eq:betafirst}) and (\ref{eq:longprimesum}) of Lemma \ref{lem:prime-sum}.
\end{proof}

\end{lemma}

For $2\leq z<w$ and $\Re(s)>0$, define
\begin{align*}
F_{(z,w]}^{\mathrm{rand}}(s)
:=\prod_{z<p\leq w}\left(1-\frac{e^{i\theta_p}}{p^s}\right)^{-1}
\left(1-\frac{e^{-i\theta_p}}{p^s}\right)^{-1}.
\end{align*}
For simplicity, we write $\exp_2(a):=\exp(e^a)$ for $a\geq 0$, and define  
$\mathcal{P}_k(a):= \{p: \exp_2(a+k-1) < p \le \exp_2(a+k)\}$.

\begin{lemma}\label{lem:ST-euler-product-2ndmom}
For any $t \in \mathbb{R}$ and $2 \le z < w$,
\begin{align}\label{eq:ST-euler-product-2ndmom}
\mathbb{E}\left[\left|F_{(z, w]}^{\mathrm{rand}}\left(\frac{1}{2}+it\right) \right|^2 \right] = \prod_{z < p \le w} \left(1 - \frac{1}{p}\right)^{-1}.
\end{align}

\noindent In particular, for any $a \ge 0$ and $0 \le b - a \le 1$, we have 
\begin{align*}
\mathbb{E}\left[\left|F_{\exp_2(a)}^{\mathrm{rand}}\left(\frac{1}{2}+it\right) \right|^2 \right] & \ll e^a \qquad
\text{and} \qquad \mathbb{E}\left[\left|F_{(\exp_2(a), \exp_2(b)]}^{\mathrm{rand}}\left(\frac{1}{2}+it\right) \right|^2 \right] \ll 1.
\end{align*}

\begin{proof}
Since
\begin{align*}
F_{P}^{\mathrm{rand}}\left(s\right)
= \prod_{p \le P} \sum_{\nu \ge 0} \frac{\mathbb{X}(p^\nu)}{p^{\nu s}},
\end{align*}

\noindent the moment formula \eqref{eq:ST-euler-product-2ndmom} follows immediately from Lemma \ref{lem:orthog}.
The remaining assertions follow from Mertens' estimate $\prod_{p \le \exp_2(a)} (1 - p^{-1})^{-1} \asymp e^a$.
\end{proof}

\end{lemma}

\section{Gaussian comparison and ballot theorem}
Motivated by the expansion in \eqref{eq:taylor-logep}, we introduce the notation
\begin{align*}
\xi_p^{(a)}(u)
& := \frac{2 \cos(\theta_p) \cos(e^{-a}u \log p)}{p^{\frac{1}{2}}}
+\frac{\cos(2\theta_p) \cos(2e^{-a}u \log p)}{p}
\end{align*}

\noindent and observe that
\begin{equation}\label{eq:xi-moment}
\begin{split}
\mathbb{E}[\xi_p^{(a)}(u)]
& = - \frac{\cos(2e^{-a} u \log p)}{2p},\\
\mathrm{Var}\left(\xi_p^{(a)}(u)\right)
& = \frac{\cos^2(e^{-a}u \log p)}{p} + O(p^{-2})
= \frac{1 + \cos(2e^{-a}u \log p)}{2p} + O(p^{-2}).
\end{split}
\end{equation}

\noindent We also let
\begin{align*}
S_0^{(a)}(u) := 0, \quad S_k^{(a)}(u) := \sum_{j=1}^k Y_j^{(a)}(u)
\quad \text{where} \quad Y_k^{(a)}(u) & := \sum_{p \in \mathcal{P}_k(a)} \xi_p^{(a)}(u),
\end{align*}

\noindent and whenever there is no risk of confusion we also suppress the dependence on $a$ in all of these expressions.

We would now like to introduce a notion that says the oscillatory terms coming from the moments of $\xi_p^{(a)}$ are negligible at suitable scales.
To do so, let us fix some sufficiently large constant $a_0$.
For any $a \ge a_0$, $N \ge 1$, and any interval $J \subset \mathbb{R}$,
we say $(a, N, J)$ is admissible if 
\begin{align} \label{eq:admissible}
\sup_{u \in J_2^+} 
\left| \sum_{p \in \mathcal{P}_k(a)} \frac{\cos(2e^{-a} u \log p)}{p}\right| 
\ll e^{-k}, \qquad 1 \le k \le N
\end{align}

\noindent where $J_2^+ := \{u \in \mathbb{R}: |u-J| \le 2\}$
and the implicit constant (which depends only on $a_0$) is fixed throughout.

\begin{lemma}\label{lem:cumulant}
Let $(a, N, J)$ be admissible.
Uniformly in $u \in J_2^+$, $1 \le k \le N$, and $|\lambda| \le L$ for fixed $L > 0$, we have
\begin{align}
\label{eq:cumulant-Y}
\log \mathbb{E}\left[e^{\lambda Y_k^{(a)}(u)}\right] & = \frac{\lambda^2}{4} + O_L(e^{-k}), \\
\label{eq:cumulant-S}
\log \mathbb{E}\left[e^{\lambda S_k^{(a)}(u)}\right] & = \frac{\lambda^2 k}{4} + O_L(1).
\end{align}

\noindent The same estimates hold after applying up to three derivatives with respect to $\lambda$.
\end{lemma}

\begin{proof}
We focus on the proof of \eqref{eq:cumulant-Y} since \eqref{eq:cumulant-S} follows by summation.
Since $|\xi_p(u)| \ll p^{-1/2}$, we obtain by Taylor expansion that
\begin{align*}
\log \mathbb{E}\left[e^{\lambda \xi_p(u)}\right]
& = \lambda \mathbb{E}[\xi_p(u)] + \frac{\lambda^2}{2} \mathrm{Var}(\xi_p(u)) + O_L(p^{-3/2})\\
& = \frac{\lambda^2}{4p} + \frac{\lambda(\lambda-2)}{4p}\cos(2e^{-a} u\log p) + O_L(p^{-3/2})
\end{align*}

\noindent and the same expansion may be differentiated three times in $\lambda$ with the same error control.
Summing over $p \in \mathcal{P}_k(a)$, the desired estimate now follows by Lemma \ref{lem:prime-sum} (using \eqref{eq:psum-needed-random1} with $\tau = 0$)
and the $(a, N, J)$-admissibility condition.
\end{proof}

\begin{lemma}[Gaussian approximation]\label{lem:gaussian_approx}
Let $(a, N, J)$ be admissible,
and write $\mu_k(u) := \mathbb{E}[Y_k^{(a)}(u)]$ and $\sigma_k^2(u) := \mathrm{Var}(Y_k^{(a)}(u))$.
There exists some $c > 0$ such that for $G_{k, u} \sim \mathcal{N}(\mu_k(u), \sigma_k^2(u))$, we have
\begin{align}\label{eq:berry-esseen}
\sup_{x \in \mathbb{R}} \left|\mathbb{P}(Y_k^{(a)}(u) \le x) - \mathbb{P}(G_{k, u} \le x) \right|
\ll \exp \left(-c e^{a+k-1}\right).
\end{align}

\noindent In particular, if $N_k \sim \mathcal{N}(0, 1/2)$, $\delta_k := (k+1)^{-4}$ and $|v| \le Ck$,
then 
\begin{align}\label{eq:gapprox-2}
\mathbb{P}(Y_k^{(a)}(u) \in [v, v+\delta_k]) \le (1+\eta_k) \mathbb{P}(N_k \in [v, v+\delta_k]),
\qquad \eta_k \ll_C k^2 e^{-k}
\end{align}

\noindent uniformly in $u \in J$ provided that $a_0$ is sufficiently large.
\end{lemma}

\begin{proof}
It is straightforward to verify that
\begin{align*}
\mu_k(u) \ll e^{-k}, \qquad \sigma_k^2(u) = \frac{1}{2} + O(e^{-k}).
\end{align*}

By the classical Berry--Esseen theorem (see e.g. \cite[Theorem 3.6]{CGS2011}), we have
\begin{align*}
\sup_{x \in \mathbb{R}} \left|\mathbb{P}(Y_k(u) \le x) - \mathbb{P}(G_{k, u} \le x) \right|
&\ll \sum_{p \in \mathcal{P}_k(a)} \mathbb{E}\left[|\xi_p(u) - \mathbb{E}[\xi_p(u)]|^3\right] \\
& \ll \sum_{p \in \mathcal{P}_k(a)} p^{-3/2} \ll \exp \left(-c e^{a+k-1}\right)
\end{align*}

\noindent which is \eqref{eq:berry-esseen}. Finally, one can compare the densities of $G_{k,u}$ and $N_k$ and check that 
\begin{align*}
\left|\log \frac{\mathbb{P}(G_{k,u} \in dx)}{\mathbb{P}(N_k \in dx)}\right|
\ll \frac{\mu_k(u)}{\sigma_k^2(u)} |x| + \frac{\mu_k(u)^2}{\sigma_k^2(u)} + |\sigma_k^2(u) - \tfrac{1}{2}| x^2
\ll_C k^2 e^{-k}, \qquad |x| \le 2Ck.
\end{align*}

\noindent Since the error in \eqref{eq:berry-esseen} is bounded by $\eta_k \mathbb{P}(N_k \in [v, v+\delta_k])$ 
for $a_0$ sufficiently large, the claim \eqref{eq:gapprox-2} now follows and our proof is complete.
\end{proof}

\begin{lemma}\label{lem:local_density}
Let $(a, N, J)$ be admissible and $r_0 > 0$ be some sufficiently large integer.
For any fixed $C>0$, we have
\begin{align}
\mathbb{P}(S_r^{(a)}(u) \in [v, v+\delta_r]) \ll_C \frac{\delta_r}{\sqrt{r}} e^{-v^2 / r}
\end{align}

\noindent uniformly in $u \in J_2^+$, $r_0 \le r \le N$, $|v| \le Cr$ and $\delta_r := (r+1)^{-4}$.
In particular, we have $\mathbb{P}(S_r^{(a)}(u) \in [v, v+\delta_r]) \ll_C \mathbb{P}(G_{r} \in [v, v+\delta_r])$ 
where $G_{r} \sim \mathcal{N}(0, r/2)$.
\end{lemma}

\begin{proof}
Write $\lambda :=2v/r$ and $K_r(\lambda) := \log \mathbb{E}[e^{\lambda S_r(u)}]$.
Then
\begin{align*}
\mathbb{P}(S_r(u) \in [v, v+\delta_r])
&\le \mathbb{E}\left[e^{\lambda(S_r(u) - v) + |\lambda| \delta_r} \mathbf{1}_{\{S_r(u) \in [v, v+\delta_r]\}}\right]\\
& \ll_C e^{-\lambda v +K_r(\lambda)} \sup_{x \in \mathbb{R}}\mathbb{P}_{\lambda}(S_r(u) \in [x, x+\delta_r])
\end{align*}
\noindent with the probability measure $\mathbb{P}_\lambda$ (and the associated expectation $\mathbb{E}_{\lambda}$) defined via $d\mathbb{P}_{\lambda} / d\mathbb{P} = e^{\lambda S_r(u) - K_r(\lambda)}$.
Under the tilted measures, the increments $\xi_p(u)$ are still independent of each other.
In particular, if we write
\begin{align*}
    S_r(u) = \sum_{j < r/2} Y_j(u) + \sum_{j = \lfloor r/2 \rfloor}^r Y_j(u) =: S_{r, -}(u) + S_{r, +}(u),
\end{align*}

\noindent we can condition on $S_{r, -}(u)$ and consider comparison against a Gaussian random variable with a matching conditional mean and variance as $S_r(u)$.
Using \eqref{eq:cumulant-Y} with two derivatives in Lemma \ref{lem:cumulant}, we see that 
\begin{align*}
    \mathbb{E}_{\lambda}[|S_{r, +}(u) - \mathbb{E}_\lambda[S_{r, +}(u)]|^2]
    = \sum_{j = \lfloor r/2 \rfloor}^r\mathbb{E}_{\lambda}[|Y_{j}(u) - \mathbb{E}_\lambda[Y_j(u)]|^2] 
    \asymp r.
\end{align*}

\noindent Combining this with the trivial third moment estimates
\begin{align*}
\sum_{\log p > e^{a+r/2}} \mathbb{E}_{\lambda}[|\xi_p(u) - \mathbb{E}_{\lambda}[\xi_p(u)]|^3]
\ll \sum_{\log p > e^{a+r/2}} p^{-3/2} \ll \exp(-ce^{a+r/2})
\end{align*}

\noindent and the usual density estimate for a Gaussian distribution, we obtain
\begin{align*}
\sup_{x \in \mathbb{R}} \mathbb{P}_\lambda(S_r(u) \in [x, x+\delta_r]) 
\ll_C \frac{\delta_r}{\sqrt{r}} + \exp(-ce^{a+r/2})
\ll_C \frac{\delta_r}{\sqrt{r}}
\end{align*}

\noindent which concludes our proof.
\end{proof}

\begin{lemma}\label{lem:maxbound}
Let $(a, N, J)$ be admissible, $I \subset J$ be an interval of length at most $e^{-k}$,
and $r_0 > 0$ be a sufficiently large integer. 
For any fixed $C > 0$, we have 
\begin{align}
\mathbb{P}(\max_{u \in I} |S_k^{(a)}(u)| > x) \ll_C e^{-x^2 / k}
\end{align}

\noindent uniformly in $r_0 \le k \le N$ and $0 \le x \le Ck$.
\end{lemma}

\begin{proof}
We will only sketch the maximal inequality for $S_k(u)$ but the same argument applies to $-S_k(u)$,
and our approach follows \cite[Appendix B]{hamdan2026partialsumsrandommultiplicative} closely.
To begin with, consider any $u_0, u_1, u_2 \in J_2^+$ satisfying
\begin{align*}
|u_1 - u_2| \le e^{-k} \qquad \text{and} \qquad D:= e^k |u_1 - u_2| \le 1
\end{align*}

\noindent and we want to show that
\begin{align}\label{eq:joint-tail}
\mathbb{P}(S_k(u_0) > x, S_k(u_1) - S_k(u_2) > z) \ll_C \exp(-x^2 /k - c z^{3/2} / D)
\end{align}

\noindent for some $c \in (0, \infty)$, uniformly in $0 \le x \le Ck$ and $0 \le z \le e^{2k}$.
Based on Chernoff bound (see the proof of \cite[Lemma B.1]{hamdan2026partialsumsrandommultiplicative}),
it suffices to show that for any $L > 0$, there exists some constant $C_L > 0$ such that
\begin{align}\label{eq:pf-laplace}
\log \mathbb{E}\left[\exp \left(\lambda_1 S_k(u_0) + \lambda_2(S_k(u_1) - S_k(u_2))\right)\right]
\le \frac{\lambda_1^2k}{4} + C_L\left[1+(\lambda_2D) + (\lambda_2D)^2\right]
\end{align}

\noindent uniformly in $0 \le \lambda_1 \le L$ and $0 \le \lambda_2|u_1 - u_2| \le 1$,
and in particular with the choice $\lambda_1 = 2x/k$ and $\lambda_2 = c \sqrt{z} / D$.
If we write
\begin{align*}
A_p := \lambda_1 \xi_p(u_0), \qquad B_p := \lambda_2 \left[\xi_p(u_1) -\xi_p(u_2)\right], \qquad W_p := A_p + B_p,
\end{align*}

\noindent then the joint Laplace transform is equal to 
\begin{align*}
\sum_{e^a < \log p \le e^{a+k}} \log \mathbb{E}[\exp(W_p)]
& = \sum_{e^a < \log p \le e^{a+k}} \left(\mathbb{E}[W_p] + \frac{1}{2}\mathrm{Var}(W_p) + O_L(\mathbb{E}[|W_p|^3])\right)
\end{align*}

\noindent by Taylor expansion. This is possible because of the deterministic bound
$|W_p| \ll (\lambda_1 +e^{-a}|u_1 - u_2|\lambda_2 \log p) / p^{1/2}$,
which also implies the third moment estimate
\begin{align*}
\sum_{e^a < \log p \le e^{a+k}} \mathbb{E}[|W_p|^3]
&\ll \sum_{e^a < \log p \le e^{a+k}} \left[\frac{\lambda_1^3}{p^{3/2}} + \frac{\left(e^{-a}|u_1 - u_2|\lambda_2 \log p\right)^3}{p^{3/2}}\right]
\ll_L 1.
\end{align*}

\noindent We then perform moment estimates using \eqref{eq:ST-moments}, the admissibility condition \eqref{eq:admissible} of $(a, N, J)$,
and the bounds $\sum_{e^a < \log p \le e^{a+k}} \log^\ell p / p \ll e^{\ell(a+k)}$ for any $\ell \ge 1$ deduced from partial summation with $\pi(x) \ll x / \log x$.
For the first two moments of $A_p$, we have:
\begin{align*}
\left| \sum_{e^a < \log p \le e^{a+k}}\mathbb{E}[A_p] \right|
& = \lambda_1 \sum_{e^a < \log p \le e^{a+k}}\frac{\cos(2e^{-a} u_0 \log p)}{p} \ll \lambda_1,\\
\mathrm{Var}\left(\sum_{e^a < \log p \le e^{a+k}}A_p\right)
& = \lambda_1^2\left[\sum_{e^a < \log p \le e^{a+k}} \left(\frac{1 + \cos(2e^{-a} u_0 \log p)}{2p} + O(p^{-2})\right)\right] \\
& = \frac{\lambda_1^2 k}{2} + O_L(1).
\end{align*}

\noindent As for $B_p$, we have
\begin{align*}
\left|\sum_{e^a < \log p \le e^{a+k}} \mathbb{E}[B_p]\right|
& = \lambda_2\left| \sum_{e^a < \log p \le e^{a+k}} \frac{\cos(2e^{-a} u_1 \log p) - \cos(2e^{-a} u_2 \log p)}{2p} \right|\\
& \ll \lambda_2e^{-a}|u_1 - u_2|\left| \sum_{e^a < \log p \le e^{a+k}} \frac{\log p}{p} \right|
\ll \lambda_2 e^{-a} e^{-k}D e^{a+k}
=\lambda_2 D
\end{align*}

\noindent and similarly
\begin{align*}
\mathrm{Var}\left(\sum_{e^a < \log p \le e^{a+k}} B_p\right)
& \ll \lambda_2^2\sum_{e^a < \log p \le e^{a+k}}\frac{|u_1-u_2|^2 e^{-2a} \log^2 p}{p} 
\ll \lambda_2^2D^2.
\end{align*}

\noindent Using $\mathbb{E}[\cos(\theta_p)\cos(2\theta_p)] = 0$, it is also straightforward to check that
\begin{align*}
\sum_{e^a < \log p \le e^{a+k}} \mathrm{Cov}( A_p, B_p)
& \le \lambda_1 \lambda_2\sum_{e^a < \log p \le e^{a+k}} \left[\frac{e^{-a}|u_1 - u_2| \log p}{p} + \frac{e^{-a}|u_1 - u_2| \log p}{p^2}\right]\\
&\ll \lambda_1 \lambda_2 e^{-a} |u_1 - u_2| \left( e^{a+k} + \exp(-e^a)\right)
\ll \lambda_1 \lambda_2 D.
\end{align*}

\noindent In other words, 
\begin{align*}
\left|\sum_{e^a < \log p \le e^{a+k}} \mathbb{E}[W_p]\right|
& = \left|\sum_{e^a < \log p \le e^{a+k}} \mathbb{E}[A_p + B_p]\right|
\ll \lambda_1 + \lambda_2 D,\\
\left|\sum_{e^a < \log p \le e^{a+k}} \mathrm{Var}(W_p)\right|
& = \left|\sum_{e^a < \log p \le e^{a+k}} \left(\mathrm{Var}(A_p) + \mathrm{Var}(B_p) +  2\mathrm{Cov}( A_p, B_p)\right)\right|\\
& = \frac{\lambda_1^2 k}{2} + O(\lambda_1\lambda_2 D + \lambda_2^2D^2).
\end{align*}

\noindent This verifies \eqref{eq:pf-laplace} and hence \eqref{eq:joint-tail}.
Finally, the desired maximal inequality follows from a standard chaining argument,
and we refer the reader to the proof of \cite[Lemma B.2]{hamdan2026partialsumsrandommultiplicative} for details
(unlike the Steinhaus case, we have to work with general intervals because our function $u \mapsto S_k(u)$ is not translation invariant in distribution,
but the admissibility of $(a, N, J)$ allows us to derive analogous estimates with the required uniformity).
\end{proof}

The arguments in Section \ref{sec:proofProp3.1} make use of a ballot theorem (Lemma \ref{lem:gaussian_ballot}) and an analogous statement for the partial sums $S_k^{(a)}(u)$ (Lemma \ref{lem:tailS_withgoodevent}). In order to state these results, we introduce the following notation. Fix a sufficiently large constant $r_0>0$. For any $A \ge 1$, write $r_A := \max(r_0, \lceil A/4 \rceil)$, and define the barriers
\begin{align}\label{eq:barriers_def}
U_A(j) := A + j + 2 \log\big(1+(j\land N-j)\big),
\quad L_A(j) := A - 20j,
\end{align}
for $j\leq N$, where $(a\land b):=\min(a,b)$.

\begin{lemma}[Gaussian ballot estimate]\label{lem:gaussian_ballot}
Let $G_j := \sum_{\ell=1}^j N_\ell$,  where $\{N_j\}_j$ is a sequence of i.i.d., centred Gaussian random variables with variance $1/2$.
Then there exists an absolute constant $C>0$ such that
\begin{align*}
\mathbb{P}(|G_N-w|\leq 3, G_j \in [L_A(j) - 3, U_A(j) + 3] 
~ \forall r_A \le j \le N) \\\ll \frac{(A+1) (U_A(N) - w + C)}{N^{3/2}} e^{-w^2 / N}
\end{align*}

\noindent uniformly for $1 \le A \le N/(100 \log (N+2))$ and $w \in [N/4, U_A(N) + 3]$.
\end{lemma}
\begin{proof}
By a union bound, it suffices to prove the bound for
\begin{align}\label{eq:gaussianballot_local}
&\mathbb{P}\big(G_N \in [u, u+1),
~ \text{and} ~\, G_j \in [L_A(j) - 3, U_A(j) + 3] 
\quad \forall r_A \le j \le N\big)
\end{align}
uniformly in $u\in \{w-3, ...,w+2\}$, and this is a 
direct consequence of \cite[Proposition 5]{ArguinBourgadeRadziwillI}. In order to apply this proposition, we first  condition on the value of $G_{r_A}$ to express the probability in \eqref{eq:gaussianballot_local} as
\begin{align}\label{eq:ballot_gaussian_integral}
    \int_{L_A(r_A)-3}^{U_A(r_A)+3} \mathbb{P}\big(\substack{ W_{s}\in [L_A(r_A+s)-3-x, U_A(r_A+s)+3-x], \forall  0\leq s\leq N-r_A\\W_{N-r_A}\in [u-x,u-x+1)\\}\big)\frac{e^{-x^2/r_A}}{\sqrt{\pi r_A}}\mathrm{d}x
\end{align}
where $W_s:= G_{s+r_A}-G_{r_A}$. We also redefine $U_A$ to be the function 
\[
    U_A(j)=A+j+2\log\Big(1+\frac{j(N-j)}{N}\Big)=: A+j +\psi(j-r_A),
\]
which has the advantage of being twice differentiable at $j=N/2$; this replacement is without loss of generality, as this definition only differs from \eqref{eq:barriers_def} by a constant which can be absorbed into the constant $C>0$ in the lemma's statement. For each $x$ in \eqref{eq:ballot_gaussian_integral},  \cite[Proposition 5]{ArguinBourgadeRadziwillI} then applies to the integrand, using 
\begin{align}
        &g(s)=\psi(s)-\psi(0)\Big(1-\frac{s}{N-r_A}\Big),\,\alpha=1-\frac{\psi(0)}{N-r_A},\,\text{ and } y=U_A(r_A)+3-x.
\end{align}
It follows that \eqref{eq:ballot_gaussian_integral} is
\begin{align*}
        &\ll \int_{L_A(r_A)-3}^{U_A(r_A)+3}\frac{(U_A(r_A)+4-x)(U_A(N)+4-u)}{(N-r_A)^{3/2}}\frac{e^{-\frac{(u-x)^2}{N-r_A}-\frac{x^2}{r_A}}}{\sqrt{ r_A}}\mathrm{d}x \\
        &\ll \frac{(U_A(r_A)-L_A(r_A)+7)(U_A(N)+4-u)}{N^{3/2}}\int_{\mathbb{R}}\frac{e^{-\frac{(u-x)^2}{N-r_A}-\frac{x^2}{r_A}}}{\sqrt{ r_A}}\mathrm{d}x\\
        &\ll \frac{(A+1) (U_A(N) - u + C)}{N^{3/2}} e^{-u^2 / N}.
\end{align*}
\end{proof}

\begin{lemma}[Ballot theorem for $S_k^{(a)}(u)$] \label{lem:tailS_withgoodevent}
Let $(a, N, J)$ be admissible, $1 \le A \le\frac{N}{100\log(N+2)}$, 
and $\mathcal{G}_A:= \{u \in J: S_k^{(a)}(u) \in [L_A(k), U_A(k)] \quad \forall r_A \le k \le N\}$.
Then there exists some $C \in (0, \infty)$ independent of $N$ such that 
\begin{align}
\mathbb{P}(S_N^{(a)}(u) >w, u \in \mathcal{G}_A) \ll \frac{(A+1) (U_A(N) - w + C)}{N^{3/2}} e^{-w^2 / N}
\end{align}

\noindent uniformly in $u \in J$ and $w \in [N/4, U_A(N)]$.
\end{lemma}
\begin{proof}
    This follows from the proof of \cite[Lemma 2.2]{hamdan2026partialsumsrandommultiplicative} (itself inspired by \cite[Section 7]{ArguinBourgadeRadziwillI}), substituting the Gaussian estimates used therein with Lemmas \ref{lem:cumulant} and \ref{lem:local_density}. We give the broad strokes. 

    We will proceed by subdividing the event $\{S_N^{(a)}(u)>w, u\in \mathcal{G}_A\}$ into disjoint events of the form $\{(S_k^{(a)}-S_{k-1}^{(a)})(u)\in [v,v+\delta_k)\}$, for an appropriate choice of mesh sizes $(\delta_k)_k$. The probability of each of these events can then be compared to a Gaussian analogue via Lemmas \ref{lem:gaussian_approx} and \ref{lem:local_density}, from which one recovers the probability in Lemma \ref{lem:gaussian_ballot}, and concludes the claim by the same lemma. 

    This will require some additional notation. Let $\delta_k:= (1+k)^{-4}$ for each $r_A\leq k\leq N$, and let $\mathcal{T}\subseteq \mathbb{R}^{N-r_A+1}$ be the set of all tuples $(v_k)_k\subseteq \delta_k \mathbb{Z}$. In what follows, we will also omit the dependence on $a$ and $u$, writing $S_k=S_k^{(a)}(u)$ for each $k$ (and similarly for $Y_k$). The event $\big\{S_N\in [j,j+1), u\in \mathcal{G}_A\big\}$ is then clearly contained in
    \[
        \bigcup_{(v_k)_k\in\mathcal{T}}\Big( \big\{S_N\in [j,j+1), u\in \mathcal{G}_A\big\}\cap \big\{S_{r_A}\in[v_{r_A},v_{r_A}+\delta_{r_A}), Y_k\in [v_k,v_k+\delta_k) ~\forall r_A<k\leq N\big\}\Big),
    \]
    and, using the fact that $\sum_k \delta_k <1$, one straightforwardly checks that any non-empty element of the union above must necessarily be associated to a tuple $(v_k)_k$ satisfying
\begin{gather}\label{eq:constraints_vk}
        \forall r_A\leq \ell\leq N,\quad L_A(\ell)-2\leq \sum_{k=r_A}^\ell v_k\leq U_A(\ell),  \\
    \Big|j-\sum_{k=r_A}^Nv_k\Big|\leq 2
\end{gather}
as well as $|v_k|\leq 100k$ for all $r_A\leq k\leq N$. Letting $\mathcal{T}'(j)\subseteq \mathcal{T}$ denote the subset of tuples in $\mathcal{T}$ satisfying these constraints, it then follows by a union bound that
\[
    \mathbb{P}\big(S_N\in [j,j+1), u\in \mathcal{G}_A\big)\leq \!\!\sum_{(v_k)_k\in \mathcal{T}'(j)}\!\! \mathbb{P}\big(S_{r_A}\in [v_{r_A},v_{r_A}+\delta_{r_A})\big)\!\!\prod_{k=r_A+1}^N\mathbb{P}\big(Y_k\in [v_k,v_k+\delta_k)\big).
\]
Applying Lemmas \ref{lem:gaussian_approx} and \ref{lem:local_density} (the latter with $C=100$) to each summand shows that the above is
\begin{align}\label{eq:mainproblemma_finalsum}
        \ll \sum_{(v_k)_k\in \mathcal{T}'(j)}\mathbb{P}\big(G_{r_A}\in [v_{r_A},v_{r_A}+\delta_{r_A})\big)\prod_{k=r_A+1}^N(1+\eta_k)\mathbb{P}\big(N_k\in [v_k,v_k+\delta_k)\big),
\end{align}
where $G_j=\sum_{\ell=1}^j N_\ell$ is a Gaussian random walk with centred, i.i.d.\ increments having variance $1/2$, and $\eta_k\ll k^2e^{-k}$. In particular, the product $\prod_k (1+\eta_k)$ can be absorbed into the implicit constant. 

To conclude, note that for any $(v_k)_k\in \mathcal{T}'(j)$, the constraints in \eqref{eq:constraints_vk} imply that
\begin{align*}
    &\{G_{r_A}\in [v_{r_A},v_{r_A}+\delta_{r_A}), \text{ and }N_k\in [v_k,v_k+\delta_k)\forall r_A< k\leq N\} \\
    &\subseteq \{|G_N-j|\leq 3, \text{ and }\forall r_A\leq k\leq N, \,L_A(k)-3\leq G_k\leq U_A(k)+3\}.
\end{align*} 
The events on the left-hand side being disjoint for different tuples $(v_k)_k$, we can therefore bound the sum in \eqref{eq:mainproblemma_finalsum} by
\[
    \mathbb{P}(|G_N-j|\leq 3, \text{ and }\forall r_A\leq k\leq N, \,L_A(k)-3\leq G_k\leq U_A(k)+3)
\]
which by Lemma \ref{lem:gaussian_ballot}, gives
\[
    \mathbb{P}(S_N\in [j,j+1), u \in \mathcal{G}_A) \ll \frac{(A+1) (U_A(N) - j + C)}{N^{3/2}} e^{-j^2 / N}
\]
uniformly in $j\in [w,  U_A(N)]\cap \mathbb{Z}$. Recalling that $w\geq N/4$ by assumption, a union bound over all such $j$ gives
\begin{align*}
    \mathbb{P}(S_N>w, u \in \mathcal{G}_A) &\ll \sum_{j\in [w,  U_A(N)]\cap \mathbb{Z}} \frac{(A+1) (U_A(N) - j + C)}{N^{3/2}} e^{-j^2 / N}\\
    &\ll \frac{(A+1)}{N^{3/2}}e^{-w^2/N}\sum_{j-w\geq 0} (U_A(N) - w- (j-w)+ C) e^{-(j-w)/2},\\
    &\ll \frac{(A+1) (U_A(N) - w + C')}{N^{3/2}} e^{-w^2 / N},
\end{align*}
for a constant $C'>0$ depending on $C$.
\end{proof}

\section{Proof of Proposition \ref{prop:random-mass-bound}}\label{sec:proofProp3.1}

\begin{lemma}\label{lem:ST-euler-mom-localised}
Let $(a, N, J)$ be admissible, and assume further that $|J|<C$ for some constant $C>0$.
Uniformly for $q \in [2/3, 1]$, we have
\begin{align*}
\mathbb{E}\left[ \left(\int_J \left| F_{(\exp_2(a), \exp_2(a+N)]}^{\mathrm{rand}} \left(\frac{1}{2} + ie^{-a}u\right)\right|^2 du\right)^q\right]
\ll_C \left(\frac{e^N}{1 + (1-q)\sqrt{N}}\right)^q.
\end{align*}

\end{lemma}
\begin{proof}
    For each $A>0$, let $\mathcal{G}_A$ be as in the statement of Lemma \ref{lem:tailS_withgoodevent} and
    \[
        Z_N(A):=\mathbf{1}_{J\subseteq \mathcal{G}_A} \int_J e^{2S_N^{(a)}(u)}\mathrm{d}u.
    \]
    Set $Z_N:=\int_J e^{2S_N^{(a)}(u)}\mathrm{d}u$, noting that
    \[
        \mathbb{E}\left[ \left(\int_J \left| F_{(\exp_2(a), \exp_2(a+N)]}^{\mathrm{rand}} \left(\frac{1}{2} + ie^{-a}u\right)\right|^2 du\right)^q\right] \ll e^{4q \sum_{p}\sum_{k\geq 3}p^{-k/2}}\mathbb{E}\big[Z_N^q\big]\ll \mathbb{E}\big[Z_N^q\big].
    \]
    Following the strategy used to prove \cite[Proposition 1.2]{hamdan2026partialsumsrandommultiplicative}, we will show that uniformly in $1\leq A\leq {3}\sqrt{N}$, $\mathbb{E}Z_N(A)\ll_C Ae^N/\sqrt{N}$ and $\mathbb{P}(\exists u\in J\setminus \mathcal{G}_A)\ll e^{-2A}$. From these two estimates, one can straightforwardly deduce that $\mathbb{E}[Z_N^q]\ll e^{qN}(1+(1-q)\sqrt{N})^{-q}$ using an interpolation argument.

    For the first claim, we use Fubini's theorem to write
    \begin{align*}
        \mathbb{E}[Z_N(A)]&\leq \mathbb{E} \int_J e^{2S_N^{(a)}(u)} \mathbf{1}(u\in \mathcal{G}_A)\mathrm{d}u \\&= 2\int_{-\infty}^{U_A(N)}\int_Je^{2V}\mathbb{P}(S_N^{(a)}(u)>V, u\in \mathcal{G}_A)\mathrm{d}V\mathrm{d}u.
    \end{align*}
    The contribution from $V\in (-\infty,N/4]$ is easily seen to be $\ll |J|e^{N/2} \ll_C e^{N/2}$, by a trivial bound on the probability in the integrand. For the remaining range, Lemma \ref{lem:tailS_withgoodevent} gives
    \begin{align*}
&\ll |J|\int_{N/4}^{U_A(N)}  \frac{e^{2V-V^2 / N}}{\sqrt{N}}\frac{(A+1) (U_A(N) - V + C)}{N}  \mathrm{d}V\\
&\ll_C A\int_{N/4}^{A+N}  \frac{e^{2V-V^2 / N}}{\sqrt{N}}\frac{(A+N - V + C)}{N}  \mathrm{d}V,
    \end{align*}
which, by the change of variables $v=({N}-V)/\sqrt{N}$ and the assumption that $A\leq 3\sqrt{N}$, is 
\begin{align*}
    \ll_C \frac{Ae^{N}}{\sqrt{N}}\int_{-A/\sqrt{N}}^{\sqrt{N}} (|v|+1)e^{-v^2}\mathrm{d}v
    & \ll \frac{Ae^N}{\sqrt{N}}.
\end{align*}

To prove the second claim, we begin with the following union bound 
\[
    \mathbb{P}\big(\exists u\in J\setminus\mathcal{G}_A\big)\leq \sum_{r_A\leq k\leq N}\bigg(\mathbb{P}\big(\exists u\in J: S_k^{(a)}(u)>U_A(k) \big)+\mathbb{P}\big(\exists u\in J: S_k^{(a)}(u)<L_A(k) \big)\bigg).
\]
For each $k$ in this sum, let $(I_\ell^{(k)})_{\ell\leq L_k}$ be a partition of $J$ into $L_k\asymp e^k$ many disjoint intervals of length $\leq e^{-k}$. Another union bound then gives
\[
    \leq \sum_{r_A\leq k\leq N} \sum_{\ell\leq L_k}\bigg(\mathbb{P}\Big(\max_{u\in I^{(k)}_\ell} \big|S_k^{(a)}(u)\big|>U_A(k) \Big)+\mathbb{P}\Big(\max_{u\in I^{(k)}_\ell} \big|S_k^{(a)}(u)\big|>-L_A(k) \Big)\bigg),
\]
and Lemma \ref{lem:maxbound} yields
\[
    \ll \sum_{r_A\leq k\leq N} e^{k}\big(e^{-U_A(k)^2/k}+e^{-L_A(k)^2/k}\big) \ll e^{-2A} \sum_{r_A\leq k\leq N}\bigg(\frac{1}{(1+(k\land N-k))^2}+e^{-399k+40A}\bigg).
\]
The remaining sum is $\ll 1$ uniformly in $1\leq A \leq 3\sqrt{N}$, since $r_A\geq A/4$.

\bigskip 

We are now equipped to show that $\mathbb{E}[Z_N^q]\ll e^{qN}(1+(1-q)\sqrt{N})^{-q}$. Note that it suffices to do so for $(1-q)\geq 1/\sqrt{N}$; the sought after bound is simply $\ll e^{Nq}$ otherwise, in which case it follows directly from Lemma \ref{lem:ST-euler-product-2ndmom} and Jensen's inequality.

Consider the sequence $(A_m)_{m\geq 1}$ defined as $A_m= m/(1-q)$, noting that $1\leq A_m\leq 3\sqrt{N}$ for each $1\leq m\leq \lceil \sqrt{N}(1-q)\rceil -1=:M$. Then we can decompose $Z_N$ iteratively as
\begin{align}\label{eq:decomposition}
    {Z}_N\leq {Z}_N(A_1)+\sum_{1\leq m\leq M} \mathbf{1}(\exists u \in J\setminus \mathcal{G}_{A_m}){Z}_N(A_{m+1})+\mathbf{1}(\exists u\in J\setminus \mathcal{G}_{A_{M+1}}){Z}_N.
\end{align}
Taking the $q$-th moment of both sides and using the subadditivity of $x\mapsto x^q$ gives
\[
    \mathbb{E}[{Z}_N^q]\leq \mathbb{E}[{Z}_N(A_1)^q]+\sum_{1\leq m\leq M} \mathbb{E}\big[\mathbf{1}(\exists u \in J\setminus \mathcal{G}_{A_m}){Z}_N(A_{m+1})^q\big]+\mathbb{E}\big[\mathbf{1}(\exists u\in J\setminus \mathcal{G}_{A_{M+1}}){Z}_N^q\big],
\]
and Hölder's inequality bounds the right-hand side by
\[
    \mathbb{E}[{Z}_N(A_1)]^q+\sum_{1\leq m\leq M} \mathbb{P}(\exists u \in J\setminus \mathcal{G}_{A_m})^{1-q}\mathbb{E}[{Z}_N(A_{m+1})]^q+\mathbb{P}(\exists u\in J\setminus \mathcal{G}_{A_{M+1}})^{1-q}\mathbb{E}[{Z}_N\big]^q.
\]
Applying the estimates we have derived to each $\mathbb{E}[Z_N(A_m)]^q$ and $\mathbb{P}(\exists u\in J\setminus \mathcal{G}_{A_m})$ and recalling that $\mathbb{E}[Z_N]^q\ll e^{Nq}$ (by Lemma \ref{lem:ST-euler-product-2ndmom}), we conclude that the above is
\[
    \ll_C e^{qN}\Big(e^{-2M}+\big(\sqrt{N}(1-q)\big)^{-q}\sum_{m\leq M} (m+1)^{q}e^{-2m}\Big) \ll_C \bigg(\frac{e^{N}}{(1-q)\sqrt{N}}\bigg)^q.
\]
\end{proof}
\begin{proof}[Proof of Proposition \ref{prop:random-mass-bound}]
For notational convenience, let $T:= \log \log y$ and $D_T(q) := 1+(1-q)\sqrt{T}$. It is sufficient to prove the estimate for $q \in [2/3, 1 - T^{-1/2}]$:
the result for smaller $q$ follows from H\"older's inequality,
whereas that for larger $q$ can be easily verified using Lemma \ref{lem:ST-euler-product-2ndmom}
because the desired bound (e.g. $q = 1$) becomes $\ll 1$.

By symmetry, it suffices to consider the integral of the Euler product on the positive real line. We first consider contributions from a neighbourhood of the origin.
Let $a_0$ be a sufficiently large integer according to Lemma \ref{lem:gaussian_approx}, and fix $a_* \ge a_0 + 10$.
By Lemma \ref{lem:ST-euler-product-2ndmom}, we have
\begin{align*}
\mathbb{E}\left[\left(\int_{0}^{e^{-(T/2 - a_* - 1)}} \left|F_{y}^{\mathrm{rand}}\Big(\frac{1}{2} + it\Big)\right|^2 dt\right)^q \right]
& \le \left(\int_{0}^{e^{-(T/2 - a_* - 1)}} \mathbb{E}\left[\left|F_{y}^{\mathrm{rand}}\Big(\frac{1}{2} + it\Big)\right|^2 \right]dt\right) ^q\\
& = \left(e^{-(T/2 - a_* - 1)} \cdot e^{T}\right)^q
\ll \left(\frac{e^T}{D_T(q)}\right)^q.
\end{align*}

Meanwhile, for any $0 \le r \le T/2 - a_*-2$, consider the factorisation
\begin{align*}
F_y^{\mathrm{rand}}(\cdot)
= F_{\exp_2(a_r)}^{\mathrm{rand}}(\cdot)
F_{(\exp_2(a_r), \exp_2(a_r + N_r)]}^{\mathrm{rand}}(\cdot)
F_{(\exp_2(a_r + N_r), y]}^{\mathrm{rand}}(\cdot)
\end{align*}

\noindent for $a_r := r+a_*$ and $N_r := \lfloor T-a_r \rfloor \asymp T$.
Since $x \mapsto x^q$ is concave, it follows from Jensen's inequality and Lemma \ref{lem:ST-euler-product-2ndmom} that 
\begin{align*}
&\mathbb{E}\left[\left(\int_{e^{-r-1}}^{e^{-r}} \left|F_{y}^{\mathrm{rand}}\Big(\frac{1}{2} + it\Big)\right|^2 dt\right)^q \right] 
\\ &\ll \mathbb{E}\left[\left(\int_{e^{-r-1}}^{e^{-r}} e^{a_r} \left|F_{(\exp_2(a_r), \exp_2(a_r + N_r)]}^{\mathrm{rand}}\left(\frac{1}{2}+it\right)\right|^2 dt\right)^q \right]\\
&= \mathbb{E}\left[\left(\int_{e^{a_*-1}}^{e^{a_*}} \left|F_{(\exp_2(a_r), \exp_2(a_r + N_r)]}^{\mathrm{rand}}\left(\frac{1}{2}+ie^{-a_r}u\right)\right|^2 du\right)^q \right]\\
& \ll \sum_{J \subset [e^{a_*-1}, e^{a_*}]} \mathbb{E}\left[\left(\int_J \left|F_{(\exp_2(a_r), \exp_2(a_r + N_r)]}^{\mathrm{rand}}\left(\frac{1}{2}+ie^{-a_r}u\right)\right|^2 du\right)^q \right]
\end{align*}

\noindent where the sum on the last line is over unit intervals of the form $J = [n, n+1]$ that intersect with $[e^{a_*-1}, e^{a_*}]$,
and there are $O(e^{a_*}) = O_{a_*}(1)$ such intervals.
We claim that $(a_r, N_r, J)$ is admissible: 
indeed this can be checked using \eqref{eq:psum-needed-random2} in Lemma \ref{lem:prime-sum} with the choice of parameters $X:= e^{a_r + k - 1}$ and $\tau = 2e^{-a_r}u$,
since $u \asymp e^{a_*}$ and the size of the prime sum is $\ll 1/(|\tau| X) \ll_{a_*} e^{-k}$.
It follows from Lemma \ref{lem:ST-euler-mom-localised} (and $N_r = T - a_r + O(1)$) that
\begin{align*}
\mathbb{E}\left[\left(\int_{e^{-r-1}}^{e^{-r}} \left|F_{y}^{\mathrm{rand}}\Big(\frac{1}{2} + it\Big)\right|^2 dt\right)^q \right] 
\ll \left(\frac{e^{T-r}}{D_T(q)}\right)^q.
\end{align*}

\noindent Using the subadditivity of $x \mapsto x^q$ again, we obtain
\begin{align*}
\mathbb{E}\left[\left(\frac{1}{\log y}\int_{0}^{1} \left|\frac{F_{y}^{\mathrm{rand}}(\frac{1}{2} + it)}{\frac{1}{2}+it}\right|^2 dt\right)^q \right]
&\ll \mathbb{E}\left[\left(e^{-T}\int_{0}^{e^{-(T/2 - a_* - 1)}} \left|F_{y}^{\mathrm{rand}}\Big(\frac{1}{2} + it\Big)\right|^2 dt\right)^q \right]\\
&  + \sum_{r=0}^{T/2 - a_*-2} \mathbb{E}\left[\left(e^{-T}\int_{e^{-r-1}}^{e^{-r}} \left|F_{y}^{\mathrm{rand}}\Big(\frac{1}{2} + it\Big)\right|^2 dt\right)^q \right]\\
&  \ll D_T(q)^{-q}.
\end{align*}

Now consider the contributions from unit intervals $[n, n+1]$ away from the origin.
Our approach here is similar but we will abuse the notation and consider the new choice of parameters
\begin{align*}
    N_n := \lfloor T - a_n \rfloor \qquad \text{with} \qquad a_n := \lceil a_* + 2 \log \log(10+n)\rceil
\end{align*}

\noindent so that $e^{a_n} \ll \log^2(10+n)$.
For $a_n \le T/2-2$, we apply Jensen's inequality as before and study
\begin{align*}
&\mathbb{E}\left[\left(\int_{n}^{n+1} \left|F_{y}^{\mathrm{rand}}\Big(\frac{1}{2} + it\Big)\right|^2 dt\right)^q \right]
\ll \mathbb{E}\left[\left(\int_{n}^{n+1} e^{a_n}\left|F_{(\exp_2(a_n), \exp_2(a_n + N_n)]}^{\mathrm{rand}}\Big(\frac{1}{2} + it\Big)\right|^2 dt\right)^q \right].
\end{align*}

\noindent If we consider a change of variable $u = e^{a_n}t$ and partition the integration domain (with respect to $u$)
into $O(e^{a_n})$ unit intervals $J$ as before, then one can check that $(a_n, N_n, J)$ is again admissible so that
\begin{align*}
\mathbb{E}\left[\left(\int_{n}^{n+1} e^{a_n}\left|F_{(\exp_2(a_n), \exp_2(a_n + N_n)]}^{\mathrm{rand}}\Big(\frac{1}{2} + it\Big)\right|^2 dt\right)^q \right]
&\ll e^{a_n} \left(\frac{e^{N_n}}{D_T(q)}\right)^q\\
\ll e^{(1-q)a_n}\left(\frac{e^{T}}{D_T(q)}\right)^q
&\ll \left(\log(10 + n)\right)^{2(1-q)} \left(\frac{e^{T}}{D_T(q)}\right)^q.
\end{align*}

As for $a_n > T/2 - 2$, we apply Jensen's inequality only and deduce the trivial bound
\begin{align*}
&\mathbb{E}\left[\left(\int_{n}^{n+1} \left|F_{y}^{\mathrm{rand}}(\frac{1}{2} + it)\right|^2 dt\right)^q \right]
\ll e^{qT}.
\end{align*}

Collecting all the terms, we have 
\begin{align*}
&\mathbb{E}\left[\left(\frac{1}{\log y}\int_{1}^{\infty} \left|\frac{F_{y}^{\mathrm{rand}}(\frac{1}{2} + it)}{\frac{1}{2} + it}\right|^2 dt\right)^q \right]
\\ & \qquad \le \sum_{\substack{n \ge 1 \\ a_n \le T/2 - 2}}n^{-2q}\mathbb{E}\left[\left(e^{-T}\int_{n}^{n+1} \left|F_{y}^{\mathrm{rand}}\Big(\frac{1}{2} + it\Big)\right|^2 dt\right)^q \right]\\
& \qquad \qquad \qquad+ \sum_{\substack{n \ge 1 \\ a_n > T/2 - 2}}n^{-2q}\mathbb{E}\left[\left(e^{-T}\int_{n}^{n+1} \left|F_{y}^{\mathrm{rand}}\Big(\frac{1}{2} + it\Big)\right|^2 dt\right)^q \right]\\
& \qquad\ll \sum_{\substack{n \ge 1 \\ a_n \le T/2 - 2}}\frac{\left(\log(10 + n)\right)^{2(1-q)}}{n^{2q}} D_T(q)^{-q}
+ \sum_{\substack{n \ge 1 \\ a_n > T/2 - 2}}n^{-2q}
\ll D_T(q)^{-q}
\end{align*}

\noindent because the sum over $a_n > T/2 - 2$ yields an error of size $\ll \exp(-O(e^{T/4})) = o(D_T(q)^{-q})$.
This concludes our proof.
\end{proof}

\section{Proof of Theorem \ref{thm:det} and Corollary \ref{cor:markov}}\label{sec:derandom}

To prove Theorem \ref{thm:det}, we adapt Harper's derandomization argument from \cite{harper2023typicalsizecharacterzeta}. This requires analogues of the even-moment estimates in \cite[Lemma 1]{harper2019partitionfunctionriemannzeta}, and of the comparison result \cite[Proposition 1]{harper2023typicalsizecharacterzeta} showing that characters behave like the (Steinhaus) random model. The proof of these analogues (Propositions \ref{prop:twistedmoments} and \ref{prop:bridge})  is complicated by the lack of complete multiplicativity and exact orthogonality of Hecke eigenvalues (see Lemma~\ref{lem:peter}).

Throughout this section, we assume $f\in \mathcal{H}_k(N)$, where $k\geq 2$ is an even integer and $N\geq 2$ is a sufficiently large prime.

\begin{proposition}[Twisted even moments] \label{prop:twistedmoments}
Given a finite set of primes $\mathcal{P}$, a non-empty subset $\mathcal{Q} \subseteq
\mathcal{P}\cup\{p^2:p\in\mathcal{P}\}$,
and a sequence of complex numbers $(a(q))_{q\in\mathcal{Q}}$, let
\begin{align*}
Q(\mathbb{X}):=
\sum_{q\in\mathcal{Q}}
\frac{a(q)\mathbb{X}(q)}{\sqrt{q}}
\end{align*}
and
\begin{gather*}
Q(f):=
\sum_{q\in\mathcal{Q}}
\frac{a(q)\lambda_f(q)}{\sqrt{q}}.
\end{gather*}
Let $x\geq 1$, and let $(c(n))_{n\leq x}$ be a sequence of complex numbers. Then for any integer $\ell \geq 0,$ we have
\begin{gather} \label{eq:randomtwisted}
\mathbb{E} |Q(\mathbb{X})|^{2\ell} \left| \sum_{n \leq x} c(n) \mathbb{X}(n) \right|^2 \ll
\left( \sum_{n \leq x} d_{\mathcal{P}}(n)|c(n)|^2 \right) \ell ! \left(26 \sum_{q \in \mathcal{Q}}  \frac{|a(q)|^2}{q}  \right)^{\ell},
\end{gather}
where $d_{\mathcal{P}}(n):=\sum_{\substack{d | n\\p | d \implies p \in \mathcal{P}}} 1.$
In particular, we have
\begin{align*}
\mathbb{E} |Q(\mathbb{X})|^{2\ell} \ll
 \ell ! \left(26 \sum_{q \in \mathcal{Q}}  \frac{|a(q)|^2}{q}  \right)^{\ell}.
\end{align*}

Moreover, let $U:=\max_{q \in \mathcal{Q}} q $. If $xU^{\ell} < N_k,$ then
\begin{gather} \label{eq:dettwisted}
\mathbb{E}^h_{f \in \mathcal{H}_k(N)} |Q(f)|^{2\ell} \left| \sum_{n \leq x} c(n) \lambda_f(n) \right|^2 \ll_{k}
\left( \sum_{n \leq x} d_{\mathcal{P}}(n)|c(n)|^2 \right) \ell ! \left(26 \sum_{q \in \mathcal{Q}}  \frac{|a(q)|^2}{q}  \right)^{\ell}.
\end{gather}
In particular, if $U^{\ell}<N_k,$ then
\begin{align*}
\mathbb{E}^h_{f \in \mathcal{H}_k(N)} |Q(f)|^{2\ell} \ll_{k} \ell ! \left(26 \sum_{q \in \mathcal{Q}}  \frac{|a(q)|^2}{q} \right)^{\ell}.
\end{align*}

\end{proposition}


\begin{proposition}[Derandomization of Sato--Tate random multiplicative functions] \label{prop:bridge}
 Given an integer $Y \geq 1,$ a sufficiently large $P,$ and sequences of complex numbers 
 \begin{align*}
  (a_{1}(p))_{p \leq P}, (a_{1}(p^2))_{p \leq P}, ..., (a_{Y}(p))_{p \leq P}, (a_{Y}(p^2))_{p \leq P}
 \end{align*}
 of modulus at most one, let
\begin{gather*}
S_i(f) :=  \sum_{\substack{p \leq P}} \Re\left(\frac{a_i(p)\lambda_f(p)}{\sqrt{p}}+\frac{a_i(p^2)(\lambda_f(p^2)-1)}{p} \right)  
\end{gather*}
and
\begin{gather*}
S_i(\mathbb{X}) := 
  \sum_{\substack{p \leq P}}  \Re
 \left(\frac{a_i(p)\mathbb{X}(p)}{\sqrt{p}}+\frac{a_i(p^2)(\mathbb{X}(p^2)-1)}{p} \right).
\end{gather*}
Given $x \geq 1,$ let $(c(n))_{n \leq x}$ be a sequence of complex numbers of modulus at most one. Let the functions $g_j,$ with associated parameters $J$ and $\delta$, be as in Lemma \ref{lem:partition}. 
Suppose $xP^{4SY}<N_k,$ where $S:=\lceil C_0 Y \delta^{-2} \log (J \log P) \rceil$ for some sufficiently large absolute constant $C_0>0.$ Then for any integers $j(1), \ldots, j(Y) \in [-J,J+1],$ we have
\begin{gather*}
\mathbb{E}^h_{f \in \mathcal{H}_k(N)}\left[\prod_{i=1}^Y
g_{j(i)} \left( S_i(f) \right) \left| \sum_{n \leq x} c(n)\lambda_f(n) \right|^2 \right]\\
= \mathbb{E} \left[\prod_{i=1}^Y
g_{j(i)} \left( S_i(\mathbb{X}) \right) 
\left| \sum_{n \leq x} c(n)\mathbb{X}(n) \right|^2\right]+O_{k} \left( \frac{x}{(J\log P)^{Y/\delta^2}} \right).
\end{gather*}
In particular, we have
\begin{align*}
\mathbb{E}^h_{f \in \mathcal{H}_k(N)} \left[\prod_{i=1}^Y
g_{j(i)} \left( S_i(f) \right) 
\right]= \mathbb{E} \left[\prod_{i=1}^Y
g_{j(i)} \left( S_i(\mathbb{X}) \right)\right] +O_{k} \left( \frac{1}{(J\log P)^{Y/\delta^2}} \right).
\end{align*}

\end{proposition}

Throughout this section, we set $K:=\lfloor\log^{1.01} P\rfloor$ for convenience.

\begin{lemma} \label{lem:ctntodisc}
Let $v \in \mathbb{R}.$ Then for any sufficiently large $P$, we have
\begin{gather*}
 \mathbb{E} \sum_{|h| \leq \frac{1}{2}K} \int_{-\frac{1}{2K}}^{\frac{1}{2K}} \left|F_{P}^{\mathrm{rand}} \left(\frac{1}{2} +iv+ i\frac{h}{K} + it \right) - F_{P}^{\mathrm{rand}} \left(\frac{1}{2} +iv+ i\frac{h}{K} \right)\right|^2 dt 
 \ll \log^{0.99}P .
 \end{gather*}

 \begin{proof}
This follows from adapting the proof of \cite[Lemma 2]{harper2023typicalsizecharacterzeta} with Lemma \ref{lem:orthog}.
 \end{proof}
\end{lemma}

\begin{lemma} \label{lem:discretemultchaos}
Let $q \in [2/3,1].$ Then for any sufficiently large $P,$ we have
\begin{align*}
\sum_{v\in\mathbb{Z}} \frac{1}{(1+|v|)^{2q}} \mathbb{E} \left( \frac{1}{K} \sum_{|h| \leq K/2} \left|F_{P}^{\mathrm{rand}}\left(\frac{1}{2} +iv+ i\frac{h}{K} \right)\right|^2 \right)^q
\ll \left(\frac{\log P}{1 + (1-q)\sqrt{\log\log P}}\right)^{q}.
\end{align*}

\begin{proof}
This follows from adapting the proof of \cite[Multiplicative Chaos Result 2]{harper2023typicalsizecharacterzeta} with Lemma \ref{lem:ctntodisc}.
\end{proof}

\end{lemma}

\begin{proof}[Proof of Theorem \ref{thm:det} assuming Propositions \ref{prop:twistedmoments} \&  \ref{prop:bridge}]
We adapt the argument in \cite[Theorem 1]{harper2023typicalsizecharacterzeta}, focusing on the modifications required upon replacing Steinhaus RMFs with Sato--Tate RMFs. It suffices to consider $q\in[2/3,1]$. Indeed, for any function $Z$ on $\mathcal{H}_k(N)$ and $0\leq q\leq 2/3$, Lyapunov's inequality gives
\begin{align*}
\mathbb{E}_{f \in \mathcal{H}_k(N)}^h|Z(f)|^{2q}
\leq \left(\mathbb{E}_{f \in \mathcal{H}_k(N)}^h|Z(f)|^{4/3}\right)^{3q/2},
\end{align*}
and the desired estimate follows from the case $q=2/3$, since $1-q\asymp 1$ in this range. Following \cite[Section 3.2]{harper2023typicalsizecharacterzeta}, by H\"older's inequality, we have
\begin{align*}
\mathbb{E}_{f \in \mathcal{H}_k(N)}^h \left|\sum_{\substack{n \leq x\\P^{+}(n) \leq x^{1/\log\log x}}}\lambda_f(n) \right|^{2q} 
\leq 
\left(\mathbb{E}_{f \in \mathcal{H}_k(N)}^h \left|\sum_{\substack{n \leq x\\P^{+}(n) \leq x^{1/\log\log x}}}\lambda_f(n) \right|^{2} \right)^q.
\end{align*}
Applying Lemma \ref{lem:largesieve} instead of the orthogonality relation, this is again $\ll \Psi(x,x^{1/\log\log x})^q,$ where
\begin{align*}
\Psi(x,y):=\#\{n\leq x:P^+(n)\leq y\}.
\end{align*}
Using standard smooth
number estimates (see, e.g., \cite[Theorem 7.6]{MR2378655}), this is 
$\ll (x (\log x)^{-c\log \log \log x})^q,$ and is therefore negligible. 
It remains to show
 \begin{align*}
\mathbb{E}_{f \in \mathcal{H}_k(N)}^h \left|\sum_{\substack{n \leq x\\P^{+}(n) > x^{1/\log\log x}}}\lambda_f(n) \right|^{2q} \ll_k
\left( \frac{x}{1+(1-q)\sqrt{\log \log P}} \right)^q.
\end{align*}

Following \cite[Section 3.3]{harper2023typicalsizecharacterzeta}, set $M:=\lfloor 2\log^{1.02}P\rfloor$ for some large $P$ to be chosen later and let
\begin{gather*}
S_h(f) :=  \sum_{\substack{p \leq P}} \Re\left(\frac{\lambda_f(p)}{p^{\frac{1}{2}+i\frac{h}{K}}}+\frac{\lambda_f(p^2)-1}{2p^{1+i\frac{2h}{K}}} \right)  
\end{gather*}
and
\begin{gather*}
S_h(\mathbb{X}) :=  \sum_{\substack{p \leq P}} \Re\left(\frac{\mathbb{X}(p)}{p^{\frac{1}{2}+i\frac{h}{K}}}+\frac{\mathbb{X}(p^2)-1}{2p^{1+i\frac{2h}{K}}} \right)
\end{gather*}
for integers $|h| \leq M.$ We also write
\begin{align*}
\sigma^{\mathrm{rand}}(\boldsymbol{j}):=\mathbb{E}\prod_{i=-M}^M g_{j(i)}(S_i(\mathbb{X}))
\end{align*}
for $\boldsymbol{j} \in [-J,J+1]^{2M+1},$ where $J$ is the parameter in Lemma \ref{lem:partition}, and
\begin{align*}
\mathbb{E}^{\boldsymbol{j}, \mathrm{rand}} W(\mathbb{X}):=
\sigma^{\mathrm{rand}}(\boldsymbol{j})^{-1} \mathbb{E} \left( W(\mathbb{X}) \prod_{i=-M}^M g_{j(i)}(S_i(\mathbb{X})) \right)
\end{align*}
for complex-valued functions $W$ whenever $\sigma^{\mathrm{rand}}(\boldsymbol{j})>0$. Terms for which $\sigma^{\mathrm{rand}}(\boldsymbol{j})=0$ are defined to be zero. Applying Lemma \ref{lem:orthog} and Proposition \ref{prop:bridge} with $Y=2M+1$ satisfying $xP^{4SY}<N_k,$ Harper's argument carries over to the Sato--Tate random multiplicative functions without modifications. Similarly, it remains to show
\begin{align*}
\sum_{\boldsymbol{j} \in [-J, J+1]^{2M+1}} \sigma^{\mathrm{rand}}(\boldsymbol{j}) \left( \mathbb{E}^{\boldsymbol{j}, \mathrm{rand}}
\left|\sum_{\substack{n \leq x \\ P^{+}(n) > x^{1/\log\log x}}} \mathbb{X}(n)
\right|^2\right)^q \ll \left( \frac{x}{1+(1-q)\sqrt{\log \log P}} \right)^q.
\end{align*}

Following \cite[Section 3.4]{harper2023typicalsizecharacterzeta} with Lemma \ref{lem:orthog}, Harper's argument carries over to the Sato--Tate random multiplicative functions without modifications whenever $P<x^{1/\log \log x}.$ Similarly, it remains to show that
\begin{gather*}
\sum_{|v| \leq \log^{0.01} P} \frac{1}{(1+|v|)^{2q}} \sum_{\boldsymbol{j}} \sigma^{\mathrm{rand}}(\boldsymbol{j}) \left(\mathbb{E}^{\boldsymbol{j}, \mathrm{rand}} \int_{v-1/2}^{v+1/2} |F_{P}^{\mathrm{rand}}(1/2 + it)|^2 dt \right)^q \\
\ll \Biggl(\frac{\log P}{1 + (1-q)\sqrt{\log\log P}}\Biggr)^{q}.
\end{gather*}
We shall adapt the argument from \cite[Section 3.5]{harper2023typicalsizecharacterzeta}. However, we are not confined to the case $v=0$
since large values of $v$ introduce additional difficulties in the setting of the Sato--Tate random multiplicative functions. 

Let $v\in\mathbb{Z}$ satisfy $|v| \leq \log^{0.01} P.$ Then by Jensen's inequality, we have
\begin{gather*}
\sum_{\boldsymbol{j}} \sigma^{\mathrm{rand}}(\boldsymbol{j}) \left(\E^{\boldsymbol{j}, \mathrm{rand}} \int_{v-1/2}^{v+1/2} |F_{P}^{\mathrm{rand}}(1/2+ it)|^2 dt \right)^q  \\
 \ll  \sum_{\boldsymbol{j}} \sigma^{\mathrm{rand}}(\boldsymbol{j})  \left( \E^{\boldsymbol{j}, \mathrm{rand}} \sum_{|h| \leq \frac{1}{2}K} \int_{-\frac{1}{2K}}^{\frac{1}{2K}} \left|\Delta F_{P}^{\mathrm{rand}}\left(\frac{1}{2}+iv+ i\frac{h}{K};it \right) \right|^2 dt \right)^q  \\
 + \sum_{\boldsymbol{j}} \sigma^{\mathrm{rand}}(\boldsymbol{j}) \left(\E^{\boldsymbol{j}, \mathrm{rand}} \frac{1}{K} \sum_{|h| \leq \frac{1}{2}K} \left|F_{P}^{\mathrm{rand}}\left(\frac{1}{2} +iv+ i\frac{h}{K}\right)\right|^2 \right)^q,
\end{gather*}
where
\begin{align*}
\Delta F_{P}^{\mathrm{rand}}\left(\frac{1}{2}+iv+ i\frac{h}{K};it \right):=
F_{P}^{\mathrm{rand}}\left(\frac{1}{2} +iv+ i\frac{h}{K}+it\right)
- F_{P}^{\mathrm{rand}}\left(\frac{1}{2} +iv+ i\frac{h}{K}\right).
\end{align*}
Applying H\"older's inequality followed by Lemma \ref{lem:ctntodisc}, the first sum here is
\begin{align*}
\leq \left( \E \sum_{|h| \leq \frac{1}{2}K} \int_{-\frac{1}{2K}}^{\frac{1}{2K}} \left| \Delta F_{P}^{\mathrm{rand}}\left(\frac{1}{2}+iv+ i\frac{h}{K};it \right) \right|^2 dt \right)^q \leq \log^{0.99q} P,
\end{align*}
which is negligible.

Similarly, let us define the (random) bad set
\begin{align*}
\mathcal{T}_v := \left\{h \in \Z : \left|F_{P}^{\mathrm{rand}}\left(\frac{1}{2} +iv+ i\frac{h}{K}\right)\right| \geq \log^{1.1}P \ \text{or} \ \left|F_{P}^{\mathrm{rand}}\left(\frac{1}{2} +iv+ i\frac{h}{K}\right)\right| \leq \frac{1}{\log^{1.1}P} \right\}.
\end{align*}
Then H\"older's inequality gives
\begin{gather}
 \sum_{\boldsymbol{j}} \sigma^{\mathrm{rand}}(\boldsymbol{j}) \left(\E^{\boldsymbol{j}, \mathrm{rand}} \frac{1}{K} \sum_{|h| \leq \frac{1}{2}K} \left|F_{P}^{\mathrm{rand}}\left(\frac{1}{2} +iv+ i\frac{h}{K}\right)\right|^2 \right)^q \nonumber \\
 \leq  \sum_{\boldsymbol{j}} \sigma^{\mathrm{rand}}(\boldsymbol{j}) \left(\E^{\boldsymbol{j}, \mathrm{rand}} \frac{1}{K} \sum_{\substack{|h| \leq \frac{1}{2}K \\ h \notin \mathcal{T}_v}} \left|F_{P}^{\mathrm{rand}}\left(\frac{1}{2} +iv+ i\frac{h}{K} \right)\right|^2 \right)^q \nonumber \\
+ \left( \frac{1}{K} \sum_{|h| \leq \frac{1}{2}K} \E \mathbf{1}_{\{h \in \mathcal{T}_v\}} \left|F_{P}^{\mathrm{rand}}\left(\frac{1}{2} +iv+ i\frac{h}{K}\right)\right|^2 \right)^q . 
\label{eq:goodbad}
\end{gather}
By Markov's inequality, we have
\begin{align*}
\E \mathbf{1}_{\{h \in \mathcal{T}_v\}} \left|F_{P}^{\mathrm{rand}}\left(\frac{1}{2} +iv+ i\frac{h}{K}\right) \right|^2
\leq  \frac{1}{\log^{1.1 \alpha} P} \cdot \E  \left|F_{P}^{\mathrm{rand}}\left(\frac{1}{2} +iv+ i\frac{h}{K}\right)\right|^{2+\alpha} + \frac{1}{\log^{2.2} P}
\end{align*}
for any $\alpha>0.$ Applying Lemma \ref{lem:beta}, this is
\begin{align} 
\ll_{\alpha}
\begin{cases}
(\log P)^{\beta^2-1.1 \alpha} 
( 
1+\min \left\{  \log P, K/\max\{1,|h|\} \right\}
)^{\beta(\beta-1)} 
+ \dfrac{1}{\log^{2.2} P} 
& \mbox{{\normalfont if $v=0,$} } \\
\hfil (\log P)^{\beta^2-1.1 \alpha} 
 (\log (3+|v|))^{O(\beta(\beta-1))}
+ \dfrac{1}{\log^{2.2} P} 
& \mbox{{\normalfont otherwise,} } 
\end{cases} 
\label{eq:upperboundcases}
\end{align}
where $\beta:=1+\frac{1}{2}\alpha.$ 

If $v=0,$ then (\ref{eq:upperboundcases}) gives
\begin{gather*}
\frac{1}{K} \sum_{|h| \leq \frac{1}{2}K} \E \mathbf{1}_{\{h \in \mathcal{T}_v\}} \left|F_{P}^{\mathrm{rand}}\left(\frac{1}{2} + i\frac{h}{K}\right)\right|^2 \\
= \frac{1}{K} \left\{\sum_{|h| \leq \log^{0.01} P} + \sum_{\log^{0.01} P<|h| \leq \frac{1}{2}K}\right\}\E \mathbf{1}_{\{h \in \mathcal{T}_v\}} \left|F_{P}^{\mathrm{rand}}\left(\frac{1}{2} + i\frac{h}{K}\right)\right|^2 \\
\ll_{\alpha} (\log P)^{2\beta^2-\beta-1.1\alpha-1}+(\log P)^{\beta^2-1.1\alpha} K^{\beta^2-\beta-1} \sum_{1\leq |h| \leq K} \frac{1}{|h|^{\beta^2-\beta}}+\frac{1}{\log^{2.2} P} \\
\ll_{\alpha} (\log P)^{1-0.1\alpha+\frac{1}{4}\alpha^2}.
\end{gather*}
Choosing any $\alpha \in (0.02,0.03),$ this is $\ll \log^{0.999} P,$ and the second term in (\ref{eq:goodbad}) is $\ll \log^{0.999q}P,$ which is negligible.

Otherwise, if $0<|v| \leq \log^{0.01} P,$ then (\ref{eq:upperboundcases}) gives
\begin{gather*}
\frac{1}{K} \sum_{|h| \leq \frac{1}{2}K} \E \mathbf{1}_{\{h \in \mathcal{T}_v\}} \left|F_{P}^{\mathrm{rand}}\left(\frac{1}{2} +iv+ i\frac{h}{K}\right)\right|^2 \\
\ll_{\alpha} (\log P)^{\beta^2-1.1 \alpha} 
 (\log \log P)^{O(\beta(\beta-1))}
+ \dfrac{1}{\log^{2.2} P}.
\end{gather*}
Similarly, this is again $\ll \log^{0.999} P$ for any $\alpha \in (0.02,0.03),$ and the second term in (\ref{eq:goodbad}) is $\ll \log^{0.999q}P,$ which is also negligible. 

For integers $|v|\leq\log^{0.01}P$ and $|h|\leq K/2$, we have $|vK+h|\leq M$ for all sufficiently large $P$. Regarding the first sum in (\ref{eq:goodbad}), we adapt the remaining part of the argument from \cite[Section 3.5]{harper2023typicalsizecharacterzeta}, which carries over to the Sato--Tate random multiplicative functions whenever $J \geq 1.2\log \log P$ and $0<\delta \leq (J\sqrt{\log P})^{-1}$ as the parameter in Lemma \ref{lem:partition} without modification. Similarly, it remains to show
\begin{gather*}
\sum_{|v| \leq \log^{0.01} P} \frac{1}{(1+|v|)^{2q}} \sum_{\boldsymbol{j}} \sigma^{\mathrm{rand}}(\boldsymbol{j}) \\ \cdot \left(\E^{\boldsymbol{j}, \mathrm{rand}} \frac{1}{K} \sum_{\substack{|h| \leq \frac{1}{2}K \\ h \notin \mathcal{T}_v}} 
\mathbf{1}_{\{|S_{vK+h}(\mathbb{X})-j(vK+h)|\leq 1\}}
\left|F_{P}^{\mathrm{rand}}\left(\frac{1}{2} +iv+ i\frac{h}{K}\right)\right|^2 \right)^q \\
\ll \left(\frac{\log P}{1 + (1-q)\sqrt{\log\log P}}\right)^{q}.
\end{gather*}

Following \cite[Section 3.6]{harper2023typicalsizecharacterzeta} closely, 
whenever $0<\delta \leq ( J\log^{1.2} P \sqrt{\log \log P} )^{-1},$ it reduces to showing
\begin{gather*}
\sum_{|v| \leq \log^{0.01} P} \frac{1}{(1+|v|)^{2q}} \mathbb{E}\left( \frac{1}{K} \sum_{|h| \leq K/2} \left|F_{P}^{\mathrm{rand}}\left(\frac{1}{2} +iv+ i\frac{h}{K}\right)\right|^2 \right)^q \\
\ll \left(\frac{\log P}{1 + (1-q)\sqrt{\log\log P}}\right)^{q},
\end{gather*}
which follows immediately from Lemma \ref{lem:discretemultchaos}. When $L_k$ is bounded, the desired estimate follows from Lemma \ref{lem:largesieve}. Otherwise, the proof is complete by choosing $P$ to be the largest number below $\exp(\log^{1/6} L_k)$ for which $\log^{0.01} P$ is an integer.
\end{proof}

We now turn to Corollary \ref{cor:markov}.

\begin{proof}[Proof of Corollary \ref{cor:markov}]
The corollary is an immediate consequence of Theorem \ref{thm:det} with Markov's inequality (see \cite[pp. 30-31]{harper2023typicalsizecharacterzeta} for details).
\end{proof}

\section{Proof of Proposition \ref{prop:twistedmoments}}

Due to the lack of complete multiplicativity of Hecke eigenvalues, the proof of \cite[Lemma~1]{harper2023typicalsizecharacterzeta} cannot be adapted, as it relies crucially on the complete multiplicativity of Dirichlet characters. Nevertheless, Hecke eigenvalues satisfy the Hecke recursion (see Lemma~\ref{lem:recursion}), which motivates the introduction of several operators and estimates for their operator norms with respect to the Hecke $s$-norms $\|\cdot\|_{s}$, defined as follows.

\begin{definition}[Hecke $s$-norm]
 Given a set of primes $\mathcal{P},$ an integer $s\geq 0,$ and a finitely supported sequence of complex numbers $\boldsymbol{c}=(c(n))_{n\geq 1}$, we define the \textit{Hecke $s$-norm} of $\boldsymbol{c}$ by
\[
\|\boldsymbol{c}\|_{\mathcal{P},s}=\|\boldsymbol{c}\|_{s}:=\left(\sum_{n\geq 1} b_s(t(n))|c(n)|^2\right)^{1/2},
\]
where $t(n):=\omega_{\mathcal{P}}(n)$ for integers $n\geq 1$, and
$b_s(t):=\binom{s+t}{t}$ for integers $s,t\geq 0$. 
\end{definition}

For later use, we record some elementary properties of the binomial coefficients $b_s(t).$

\begin{lemma}  \label{lem:binomial}
Let $s,t \geq 0$ be integers. Then
\begin{align*} 
tb_{s}(t)  \leq (t+1)b_s(t+1)=(s+1)b_{s+1}(t).
\end{align*}

\begin{proof}
The lemma follows by a straightforward verification.
\end{proof}

\end{lemma}

We introduce three operators $\mathcal{L}_r, \mathcal{D}, \mathcal{R}_r$ and estimate their operator norms with respect to $\|\cdot\|_s.$

\begin{lemma} \label{lem:lowering}

Given an integer $r \geq 1$ and finitely supported sequences of complex numbers $\boldsymbol{\gamma}=(\gamma_p)_{p\in\mathcal{P}}, \boldsymbol{c}=(c(n))_{n\geq 1},$ we define
\begin{align*}
(\mathcal{L}_r \boldsymbol{c})(m):=\sum_{p\in\mathcal{P}} \gamma_p c(mp^r)
\end{align*}
for integers $m \geq 1.$ Let $s \geq 0$ be an integer. Then
\begin{align*}
\|\mathcal{L}_r \boldsymbol{c}\|_s^2 \leq (s+1) \|\boldsymbol{\gamma}\|_2^2 \|\boldsymbol{c}\|_{s+1}^2,
\end{align*}
where
\begin{align*}
\|\boldsymbol{\gamma}\|_2:=\left( \sum_{p\in\mathcal{P}} |\gamma_p|^2 \right)^{1/2}.
\end{align*}
\begin{proof}
By the Cauchy--Schwarz inequality, we have
\begin{align*}
\|\mathcal{L}_r \boldsymbol{c}\|_s^2 =& \sum_m b_s(t(m))\left|\sum_{p\in\mathcal{P}} \gamma_p c(mp^r)\right|^2 \\
\leq & \|\boldsymbol{\gamma}\|_2^2 \sum_{m} b_s(t(m)) \sum_{p\in\mathcal{P}} |c(mp^r)|^2 \\
=& \|\boldsymbol{\gamma}\|_2^2 \sum_n |c(n)|^2 \sum_{\substack{p\in\mathcal{P}\\p^r | n}} b_s(t(n/p^r)).
\end{align*}
Since $t(n/p^{r}) \leq t(n)$ if $p^r \mid n$ and $b_s$ is increasing in $t$, we have $b_s(t(n/p^r)) \leq b_s(t(n)).$ Therefore, this is
\begin{align*}
\leq \|\boldsymbol{\gamma}\|_2^2 \sum_n |c(n)|^2 t(n) b_s(t(n)).
\end{align*}
Applying Lemma \ref{lem:binomial}, this is
\begin{align*}
\leq (s+1)\|\boldsymbol{\gamma}\|_2^2 \sum_n b_{s+1}(t(n)) |c(n)|^2=  (s+1)\|\boldsymbol{\gamma}\|_2^2 \|\boldsymbol{c}\|_{s+1}^2,
\end{align*}
and the lemma follows.
\end{proof}

\end{lemma}

\begin{lemma} \label{lem:diag}
Given finitely supported sequences of complex numbers $\boldsymbol{\gamma}=(\gamma_p)_{p\in\mathcal{P}}$ and $ \boldsymbol{c}=(c(n))_{n\geq 1},$ we define
\begin{align*}
(\mathcal{D} \boldsymbol{c})(m):=c(m)\sum_{\substack{p\in\mathcal{P}\\p \mid m}} \gamma_p    
\end{align*}
for integers $m \geq 1.$ Let $s \geq 0$ be an integer. Then
\begin{align*}
\|\mathcal{D} \boldsymbol{c}\|_s^2 \leq (s+1) \|\boldsymbol{\gamma}\|_2^2\|\boldsymbol{c}\|_{s+1}^2.
\end{align*}

\begin{proof}
By the Cauchy--Schwarz inequality, we have
\begin{align*}
\|\mathcal{D} \boldsymbol{c}\|_s^2 =& \sum_n b_s(t(n))\left| c(n)\sum_{\substack{p\in\mathcal{P}\\p \mid n}} \gamma_p \right|^2 \\
\leq & \|\boldsymbol{\gamma}\|_2^2 \sum_n  t(n) b_s(t(n)) |c(n)|^2.
\end{align*}
Applying Lemma \ref{lem:binomial}, this is
\begin{align*}
\leq (s+1) \|\boldsymbol{\gamma}\|_2^2 \sum_{n} b_{s+1}(t(n))|c(n)|^2 = 
(s+1) \|\boldsymbol{\gamma}\|_2^2 \|\boldsymbol{c}\|_{s+1}^2,
\end{align*}
and the lemma follows.
\end{proof}

\end{lemma}

\begin{lemma} \label{lem:raising}
Given an integer $r \geq 1$ and finitely supported sequences of complex numbers $\boldsymbol{\gamma}:=(\gamma_p)_{p\in\mathcal{P}}, \boldsymbol{c}=(c(n))_{n\geq 1},$ we define
\begin{align*}
(\mathcal{R}_r \boldsymbol{c})(m):=\sum_{\substack{p\in\mathcal{P}\\p^r | m}} \gamma_p c(m/p^r)
\end{align*}
for integers $m \geq 1.$ Let $s \geq 0$ be an integer. Then
\begin{align*}
\|\mathcal{R}_r \boldsymbol{c}\|_s^2 \leq (s+1) \|\boldsymbol{\gamma}\|_2^2 \|\boldsymbol{c}\|_{s+1}^2.
\end{align*}

\begin{proof}
By the Cauchy--Schwarz inequality, we have
\begin{align}
\|\mathcal{R}_r \boldsymbol{c}\|_s^2 = & \sum_m b_s(t(m))\left| \sum_{\substack{p\in\mathcal{P}\\p^r | m}} \gamma_p c(m/p^r) \right|^2 \nonumber\\
\leq &  \sum_m  t(m) b_s(t(m)) \sum_{\substack{p\in\mathcal{P}\\p^r | m}} |\gamma_p|^2 |c(m/p^r)|^2 \nonumber\\
=& \sum_{p\in\mathcal{P}}  |\gamma_p|^2 \sum_n t(np^r)b_s(t(np^r)) |c(n)|^2. 
\label{eq:raiseexpress}
\end{align}
Considering separately the cases $p\mid n$ and $p\nmid n$, Lemma \ref{lem:binomial} gives
\begin{align*}
t(np^r)\, b_s(t(np^r)) \leq (s+1)b_{s+1}(t(n)).
\end{align*}
Therefore, it follows from (\ref{eq:raiseexpress}) that
\begin{align*}
\|\mathcal{R}_r \boldsymbol{c}\|_s^2 \leq (s+1)\|\boldsymbol{\gamma}\|_2^2 \sum_n b_{s+1}(t(n)) |c(n)|^2
=(s+1)\|\boldsymbol{\gamma}\|_2^2 \|\boldsymbol{c}\|_{s+1}^2,
\end{align*}
and the lemma follows.
\end{proof}

\end{lemma}

\begin{proof}[Proof of Proposition \ref{prop:twistedmoments}]
We prove only \eqref{eq:dettwisted}, as the proof of \eqref{eq:randomtwisted} is largely analogous and, in fact, much simpler. The case $\ell=0$ follows directly from Lemma \ref{lem:largesieve}, so we assume that $\ell\geq 1.$
Let $\mathcal{P}_1:=\{ p \in \mathcal{P} \,:\, p \in  \mathcal{Q} \}$ and $ \mathcal{P}_2:=
\{ p \in \mathcal{P} \,:\, p^2 \in  \mathcal{Q} \}.$ We also let
\begin{align*}
\alpha_p:=\frac{a(p)}{\sqrt{p}}\cdot 1_{p \in \mathcal{P}_1},
\end{align*}
and
\begin{align*}
\beta_p:=\frac{a(p^2)}{p}\cdot 1_{p \in \mathcal{P}_2}.
\end{align*}
Then
\begin{align*}
Q(f)=\sum_{p} (\alpha_p \lambda_f(p)+\beta_p \lambda_f(p^2)).
\end{align*}
Given a finitely supported sequence of complex numbers $\boldsymbol{c}=(c(n))_{n \geq 1},$ we define
\begin{align*}
C(f):=\sum_{n} c(n) \lambda_f(n).
\end{align*}
Applying (\ref{eq:dethecke}) from Lemma \ref{lem:recursion}, we obtain
\begin{align*}
\lambda_f(m)\lambda_f(p)=\lambda_f(mp)+1_{p \mid m}\lambda_f(m/p)
\end{align*}
and
\begin{align*}
\lambda_f(m)\lambda_f(p^2)=\lambda_f(mp^2)+1_{p \mid m} \lambda_f(m)+1_{p^2 | m} \lambda_f(m/p^2)
\end{align*}
for integers $1 \leq m<N$ and primes $p<N$ in the first identity, and for integers $1\leq m<N$ and primes satisfying $p^2<N$ in the second identity. Thus
\begin{align*}
Q(f)C(f)=\sum_{m} (\mathcal{T} \boldsymbol{c})(m) \lambda_f(m),
\end{align*}
where $(\mathcal{T} \boldsymbol{c})(m)$
\begin{align}
:=
\sum_p \beta_pc(mp^2) 
+\sum_p \alpha_p c(mp)
+\sum_{p | m} \beta_p c(m)
+\sum_{p \mid m} \alpha_pc(m/p)
+\sum_{p^2 | m} \beta_pc(m/p^2) \label{eq:tau}
\end{align}
for integers $m \geq 1.$ Let $\boldsymbol{c}^{(0)}:=\boldsymbol{c}$ and $\boldsymbol{c}^{(j)}:=\mathcal{T}\boldsymbol{c}^{(j-1)}$ for integers $j \geq 1.$ Then

\begin{align}
Q(f)^j C(f) = \sum_{m} c^{(j)} (m) \lambda_f(m) \label{eq:qfjcf}
\end{align}
for $0\leq j\leq\ell$, since $\boldsymbol{c}^{(j)}$ is supported on $m\leq xU^j$ and $xU^\ell<N_k\leq N.$ Applying Lemma \ref{lem:lowering} and Lemma \ref{lem:raising} with
\begin{align*}
    \boldsymbol{\gamma}=
    \begin{cases}
    \boldsymbol{\alpha} & \mbox{{\normalfont if $r=1,$}} \\
    \boldsymbol{\beta} & \mbox{{\normalfont if $r=2$}}
    \end{cases}
\end{align*}
and Lemma \ref{lem:diag} with $\boldsymbol{\gamma}=\boldsymbol{\beta},$ 
it follows from (\ref{eq:tau}) that
\begin{align*}
\mathcal{T}=
\mathcal{L}_2
+\mathcal{L}_1
+\mathcal{D}
+\mathcal{R}_1
+\mathcal{R}_2,
\end{align*}
and
\begin{align} \label{eq:normbound}
\|\mathcal{T} \boldsymbol{c}\|_s \leq& \|\mathcal{L}_2 \boldsymbol{c}\|_s+\|\mathcal{L}_1 \boldsymbol{c}\|_s+\|\mathcal{D} \boldsymbol{c}\|_s+\|\mathcal{R}_1 \boldsymbol{c}\|_s+\|\mathcal{R}_2 \boldsymbol{c}\|_s \nonumber\\
\leq & (s+1)^{1/2}  (2\|\boldsymbol{\alpha}\|_2+3\|\boldsymbol{\beta}\|_2) \|\boldsymbol{c}\|_{s+1}.
\end{align}
Applying (\ref{eq:normbound}) successively with $s=0,1,\ldots, \ell-1,$ we obtain
\begin{align*}
\|\mathcal{T}^{\ell }\boldsymbol{c}\|_0^2 \leq & \,
(2\|\boldsymbol{\alpha}\|_2+3\|\boldsymbol{\beta}\|_2)^2 \|\mathcal{T}^{\ell-1} \boldsymbol{c}\|_1^2 \\
\leq & \,2 (2\|\boldsymbol{\alpha}\|_2+3\|\boldsymbol{\beta}\|_2)^4  \|\mathcal{T}^{\ell-2} \boldsymbol{c}\|_2 ^2 \\
\leq & \,\cdots \\
\leq & \, \ell! (2\|\boldsymbol{\alpha}\|_2+3\|\boldsymbol{\beta}\|_2)^{2\ell} \|\boldsymbol{c}\|_{\ell}^2. 
\end{align*}
Since the Hecke $0$-norm is the ordinary $\ell^2$-norm, we conclude that
\begin{align*}
\|\boldsymbol{c}^{(\ell)} \|_{\ell^2}^2:=\sum_m|c^{(\ell)}(m)|^2
=\|\boldsymbol{c}^{(\ell)}\|_0^2 \leq \ell! (2\|\boldsymbol{\alpha}\|_2+3\|\boldsymbol{\beta}\|_2)^{2\ell}
\sum_n \binom{\ell+\omega_{\mathcal{P}}(n)}{\ell} |c(n)|^2.
\end{align*}
Since
\begin{align*}
\binom{\ell+\omega_{\mathcal{P}}(n)}{\ell}
\leq 2^{\ell+\omega_{\mathcal{P}}(n)}
\leq 2^\ell d_{\mathcal{P}}(n)
\end{align*}
and $(2u+3v)^2\leq 13(u^2+v^2)$ for $u,v\geq 0,$ this is
\begin{align} \label{eq:clnorm}
\leq \ell! (26(\|\boldsymbol{\alpha}\|_2^2+\|\boldsymbol{\beta}\|_2^2))^{\ell}
\sum_{n \leq x} d_{\mathcal{P}}(n) |c(n)|^2.
\end{align}
Using (\ref{eq:qfjcf}), we have
\begin{align*}
|Q(f)|^{2\ell} |C(f)|^2 = \left| \sum_{m \leq xU^{\ell}} c^{(\ell)}(m)\lambda_f(m) \right|^2.
\end{align*}
Finally, the proposition follows from Lemma \ref{lem:largesieve} with (\ref{eq:clnorm}).
\end{proof}

\section{Proof of Proposition \ref{prop:bridge}}

We require the following estimates for twisted (mixed) even moments.

\begin{lemma} \label{lem:mixedmoments}
Given real numbers $x \geq 1$ and $P \geq 2$, let $\mathcal{P}:=\{p:p\leq P\}$. Given a sequence of complex numbers $\boldsymbol{c}=(c(n))_{n \leq x}$ of modulus at most one, let
\begin{align*}
C(\mathbb{X}):=\sum_{n \leq x} c(n) \mathbb{X}(n),
\end{align*}
\begin{align*}
C(f):=\sum_{n \leq x} c(n) \lambda_f(n),
\end{align*}
and
\begin{align*}
D:=\sum_{n \leq x} d_{\mathcal{P}}(n).
\end{align*}
Given an integer $Y \geq 1$ and sequences of complex numbers $(a_i(p))_{p \leq P}, (a_i(p^2))_{p \leq P}$ of modulus at most one, let
\begin{align*}
Q_i(\mathbb{X}):=\sum_{\substack{p \leq P}} \Re\left(\frac{a_i(p)\mathbb{X}(p)}{\sqrt{p}}+\frac{a_i(p^2) \mathbb{X}(p^2)}{p} \right)
\end{align*}
and
\begin{align*}
Q_i(f):=\sum_{\substack{p \leq P}} \Re\left(\frac{a_i(p)\lambda_f(p)}{\sqrt{p}}+\frac{a_i(p^2)\lambda_f(p^2)}{p} \right)
\end{align*}
for $i=1,\ldots,Y.$ Then for any integers $\ell_1, \ldots, \ell_Y \geq 0,$ we have
\begin{align} \label{eq:ranmixed}
\mathbb{E}  |C(\mathbb{X})|^2 \prod_{i=1}^Y |Q_i(\mathbb{X})|^{2\ell_i}
\ll D L! M^{L},
\end{align}
where $L:=\sum_{i=1}^Y \ell_i$ and $M:=26 \log \log (100P).$

Moreover, if $xP^{2L} < N_k,$ then
\begin{align} \label{eq:detmixed}
\mathbb{E}_{f \in \mathcal{H}_k(N)}^h |C(f)|^2 \prod_{i=1}^Y |Q_i(f)|^{2\ell_i}
\ll_{k} D L! M^{L}.
\end{align}

\begin{proof}
The case $L=0$ follows from Lemma \ref{lem:orthog} in the random case and Lemma \ref{lem:largesieve} in the deterministic case. We may therefore assume that $L\geq 1.$
By the weighted arithmetic--geometric mean inequality, we have 
\begin{align*}
\prod_{i=1}^Y |Q_i(\cdot)|^{2\ell_i} \leq \sum_{i=1}^Y \frac{\ell_i}{L} \cdot |Q_i(\cdot)|^{2L}.
\end{align*}
Therefore, using the bound
\begin{align*}
\sum_{p \leq P} \frac{1}{p} +\sum_{p \leq P}\frac{1}{p^2} \leq \log \log (100P), 
\end{align*}
the lemma follows from Proposition \ref{prop:twistedmoments}.
\end{proof}
\end{lemma}

We now prove Proposition \ref{prop:bridge}. The approach is similar to the one used to prove \cite[Proposition~1]{harper2023typicalsizecharacterzeta}, but the lack of exact orthogonality in our setting results in additional error terms that must be controlled. This is achieved by exploiting the Petersson trace formula, in the form of the large sieve inequality Lemma \ref{lem:largesieveineq2}.

\begin{proof}[Proof of Proposition \ref{prop:bridge}]
We first prove the case $k \geq 4.$
Let $1 \leq i \leq Y$ be fixed for now. Write $S_i(\cdot)=Q_i(\cdot)-b_i,$ where
\begin{align*}
b_i:=\Re\sum_{p \leq P} \frac{a_i(p^2)}{p}.
\end{align*} 
Let $h_i(t):=g_{j(i)}(t-b_i),$ so that
$g_{j(i)}(S_i(\cdot))=h_i(Q_i(\cdot)).$ Write $h_i=\widetilde{h_i}+r_i,$ where
\begin{align*}
\widetilde{h_i}(t)=\sum_{\ell=0}^{2S-1} d_{i,\ell} t^{\ell},
\end{align*}
where $S \geq 1$ is the integer in the statement, and
\begin{align*}
d_{i,\ell}:=\frac{h_i^{(\ell)}(0)}{\ell !} =\frac{g_{j(i)}^{(\ell)}(-b_i)}{\ell !}
\end{align*}
for $\ell = 0, \ldots, 2S-1.$ Applying Lemma \ref{lem:partition}, we have $|d_{i,0}| \leq 1$ and
\begin{align} \label{eq:diell}
|d_{i,\ell}| \leq \frac{2J+1}{\pi(\ell+1)!} \left( \frac{2\pi}{\delta} \right)^{\ell+1} 
\end{align}
for $\ell = 1, \ldots, 2S-1.$ Note that the factor $2J+1$ here accounts for the case where
$g_{j(i)} = g_{J+1} = 1 - \sum_{|j| \leq J} g_j$. Also, Taylor's theorem gives
\begin{align} \label{eq:taylorremainder}
|r_i(t)| \leq \frac{2J+1}{\pi(2S+1)!} \left( \frac{2\pi}{\delta} \right)^{2S+1}|t|^{2S}
\end{align}
for $t \in \mathbb{R}.$

Our argument diverges from that of \cite[Proposition~1]{harper2023typicalsizecharacterzeta} from this point onward, due to the lack of exact orthogonality among Hecke eigenvalues.
We begin by controlling the contribution from the main terms in the Taylor expansion. Given integers $0\leq\ell_1, \ldots, \ell_Y\leq 2S-1,$ write $\ell_i=u_i+v_i$ for some integers $u_i, v_i \geq 0$ satisfying $|U-V| \leq 1,$ where $U:=\sum_{i=1}^Y u_i$ and $V:=\sum_{i=1}^Y v_i.$ Note that $\max\{U, V\} \leq SY.$
Using the notation of Lemma \ref{lem:mixedmoments}, we define
\begin{align*}
A(\cdot) := C(\cdot) \prod_{i=1}^Y Q_i(\cdot)^{u_i}
\end{align*}
and
\begin{align*}
B(\cdot) := C(\cdot) \prod_{i=1}^Y Q_i(\cdot)^{v_i}.
\end{align*}
Applying Lemma \ref{lem:recursion}, there exist sequences of complex numbers $(\alpha_m)_{m \geq 1}, (\beta_n)_{n \geq 1}$ such that
\begin{align*}
A(f)=\sum_{m \leq xP^{2U}} \alpha_m \lambda_f(m), \qquad A(\mathbb{X})=\sum_{m \leq xP^{2U}} \alpha_m \mathbb{X}(m),
\end{align*}
and
\begin{align*}
B(f)=\sum_{n \leq xP^{2V}} \beta_n \lambda_f(n) , \qquad
B(\mathbb{X})=\sum_{n \leq xP^{2V}} \beta_n \mathbb{X}(n).
\end{align*}
Applying Lemma \ref{lem:orthog} followed by (\ref{eq:ranmixed}) of Lemma \ref{lem:mixedmoments}, we have
\begin{align} \label{eq:alpha2}
\|\alpha\|_2^2=\mathbb{E} \left| C(\mathbb{X}) \prod_{i=1}^Y Q_i(\mathbb{X})^{u_i} \right|^2 \ll D U! M^{U}, 
\end{align}
and similarly
\begin{align} \label{eq:beta2}
\|\beta\|_2^2 \ll D V! M^{V}.
\end{align}
Since $Q_i(f)$ are real, we have
\begin{align*} 
A(\cdot)\overline{B(\cdot)} =| C(\cdot) |^2 \prod_{i=1}^Y Q_i(\cdot)^{\ell_i}.
\end{align*}
Applying Lemma \ref{lem:largesieve} with (\ref{eq:alpha2}) and (\ref{eq:beta2}), we obtain
\begin{gather*} 
\mathbb{E}^h_{f \in \mathcal{H}_k(N)} |C(f)|^2 \prod_{i=1}^Y Q_i(f)^{\ell_i}
- \mathbb{E} |C(\mathbb{X})|^2 \prod_{i=1}^Y Q_i(\mathbb{X})^{\ell_i} 
 \\
\ll_k \left(\frac{xP^{2SY}}{N} \right)^{k-1}  D \sqrt{L!} M^{L/2},
\end{gather*}
where $L:=\sum_{i=1}^Y \ell_i.$ Therefore, we conclude that
\begin{gather}
\mathbb{E}^h_{f \in \mathcal{H}_k(N)} \prod_{i=1}^Y
\widetilde{h_{i}} \left( Q_i(f) \right) \left| \sum_{n \leq x} c(n)\lambda_f(n) \right|^2 
- \mathbb{E} \prod_{i=1}^Y
\widetilde{{h}_{i}} \left( Q_i(\mathbb{X}) \right) 
\left| \sum_{n \leq x} c(n)\mathbb{X}(n) \right|^2  \nonumber \\ 
\ll_k \left( \frac{xP^{2SY}}{N} \right)^{k-1}D
\sum_{\substack{0\leq\ell_i\leq 2S-1\\1\leq i\leq Y}}
\sqrt{(\ell_1+\cdots+\ell_Y)!}\,
M^{(\ell_1+\cdots+\ell_Y)/2}\prod_{i=1}^Y|d_{i,\ell_i}|. \label{eq:widetildecompare}
\end{gather}
Writing $L:=\ell_1+\cdots+\ell_Y$ and using the bound
\begin{align*}
\frac{L!}{\ell_1! \cdots \ell_{Y}!} \leq Y^L,
\end{align*}
this is
\begin{align} \label{eq:aftermultinomial}
\ll_k D  \left( \frac{xP^{2SY}}{N} \right)^{k-1}\prod_{i=1}^Y 
\left( 
\sum_{\ell_i=0}^{2S-1} |d_{i,\ell_i}| \sqrt{\ell_i!} (M Y)^{\ell_i/2}
\right).
\end{align}
Let $W:=\frac{2\pi}{\delta} \sqrt{M Y}.$ Using (\ref{eq:diell}) followed by the Cauchy--Schwarz inequality, we obtain
\begin{align*}
\sum_{\ell=0}^{2S-1} |d_{i,\ell}| \sqrt{\ell!} (M Y)^{\ell/2}
\leq& 1+\frac{2(2J+1)}{\delta} \sum_{\ell=0}^{2S-1} \frac{W^{\ell}}{\sqrt{\ell!}} \\
\leq& 1+\frac{2(2J+1)}{\delta} \cdot \sqrt{2S} 
\left( \sum_{\ell=0}^{2S-1} \frac{W^{2\ell}}{\ell!} \right)^{1/2} \\
\leq& 1+\frac{2J+1}{\delta} \cdot \sqrt{8S} \exp(W^2/2). 
\end{align*}
Combining these estimates, we conclude that
\begin{gather}
\mathbb{E}^h_{f \in \mathcal{H}_k(N)} \prod_{i=1}^Y
\widetilde{h_{i}} \left( Q_i(f) \right) \left| \sum_{n \leq x} c(n)\lambda_f(n) \right|^2 
- \mathbb{E} \prod_{i=1}^Y
\widetilde{{h}_{i}} \left( Q_i(\mathbb{X}) \right) 
\left| \sum_{n \leq x} c(n)\mathbb{X}(n) \right|^2  \nonumber \\ 
\ll_k  \left( \frac{xP^{2SY}}{N} \right)^{k-1} D Z^Y,
\label{eq:taylormaindiff}
\end{gather}
where 
\begin{align*}
Z:=1+\frac{2J+1}{\delta} \cdot \sqrt{8S} \exp \left( \frac{2\pi^2 M Y}{\delta^2} \right).
\end{align*}

When $k=2$, the first case of Lemma \ref{lem:largesieve} replaces the factor $(xP^{2SY}/N)^{k-1}$ in the preceding argument by
\begin{align*}
\frac{xP^{2SY}\log(2xP^{2SY})}{N}.
\end{align*}
Under the assumption $xP^{4SY}<N_2=N/\log N$, this is $\ll P^{-2SY}$. Thus, in the case $k=2$, the right-hand side of (\ref{eq:taylormaindiff}) is replaced by $O_2(DP^{-2SY}Z^Y)$.

It remains to control the contribution from the remainder terms in the Taylor expansion. Since
$0 \leq h_i(t)=g_{j(i)}(t-b_i) \leq 1,$ it follows from (\ref{eq:taylorremainder}) that
\begin{align*}
|\widetilde{h_i}(t)| = |h_i(t)-r_i(t)| \leq 1+ H_S |t|^{2S},
\end{align*}
where
\begin{align*}
H_S:=  \frac{2J+1}{\pi(2S+1)!} \left( \frac{2\pi}{\delta} \right)^{2S+1}.
\end{align*}
Expanding the product telescopically, we obtain
\begin{align*}
\left| \prod_{i=1}^Y h_{i}(Q_i) - \prod_{i=1}^Y \widetilde{h_{i}}(Q_i) \right|  
\leq&  \sum_{i=1}^Y |r_i(Q_i)| \prod_{j<i} |\widetilde{h_{j}}(Q_{j})| \prod_{j>i}|h_{j}(Q_{j})| \\
\leq&   \sum_{i=1}^Y H_S |Q_i|^{2S} \prod_{j < i}(1+H_S|Q_j|^{2S}).
\end{align*}
Expanding the product followed by Lemma \ref{lem:mixedmoments}, we have
\begin{gather}
\left|\mathbb{E}^h_{f \in \mathcal{H}_k(N)} \prod_{i=1}^Y
h_{i} \left( Q_i(f) \right) \left| \sum_{n \leq x} c(n)\lambda_f(n) \right|^2 
-\mathbb{E}^h_{f \in \mathcal{H}_k(N)} \prod_{i=1}^Y
\widetilde{h_{i}} \left( Q_i(f) \right) \left| \sum_{n \leq x} c(n)\lambda_f(n) \right|^2
\right| \nonumber \\
\leq \mathbb{E}^h_{f \in \mathcal{H}_k(N)} 
\left| \prod_{i=1}^Y h_{i}(Q_i(f)) - \prod_{i=1}^Y \widetilde{h_{i}}(Q_i(f)) \right|  
\left| \sum_{n \leq x} c(n)\lambda_f(n) \right|^2 \nonumber \\
\ll_k D  \sum_{i=1}^Y \sum_{j=1}^{i} \binom{i-1}{j-1} H_S^j (jS)! M^{jS}.
\label{eq:taylorremdiff}
\end{gather}
By Stirling's formula, we have
\begin{align*}
(2S)! \geq \left( \frac{2S}{e} \right)^{2S}
\end{align*}
and
\begin{align*}
(jS)! \ll \sqrt{jS} \left( \frac{jS}{e} \right)^{jS},
\end{align*}
and therefore
\begin{align} \label{eq:jS}
H_S^j (jS)! M^{jS} \ll \sqrt{jS} \left( \frac{2J+1}{\delta S} \right)^j 
\left( \frac{\pi^2 e j M}{\delta^2 S} \right)^{jS}
\end{align}
for $j=1,\ldots,Y.$ Recall that $S=\lceil C_0 Y \delta^{-2} \log (J \log P) \rceil$ for some absolute constant $C_0>0$ to be chosen later. Set $C_1:=1000.$ Since $j \leq Y$ and $M=26\log\log (100P),$ we have
\begin{align*}
\frac{\pi^2 e j M}{\delta^2 S} \leq \frac{C_1}{C_0}.
\end{align*}
Also, since $\log (2J+1) \ll \log(J\log P)$ and $\log(1/\delta) \ll \delta^{-2}$ for $\delta>0$ sufficiently small, by choosing $C_0$ sufficiently large, it follows from (\ref{eq:jS}) that $H_S^j (jS)! M^{jS}$ is
\begin{gather*}
 \ll 
\exp\left( \frac{1}{2}\log(jS) + j \log(2J+1) +j \log \left(\frac{1}{\delta} \right) - jS \log \left( \frac{C_0}{C_1} \right)  \right) 
\ll \exp(-2jS).
\end{gather*}
Substituting this into (\ref{eq:taylorremdiff}), we obtain
\begin{gather}
\mathbb{E}^h_{f \in \mathcal{H}_k(N)} \prod_{i=1}^Y
h_{i} \left( Q_i(f) \right) \left| \sum_{n \leq x} c(n)\lambda_f(n) \right|^2 
-\mathbb{E}^h_{f \in \mathcal{H}_k(N)} \prod_{i=1}^Y
\widetilde{h_{i}} \left( Q_i(f) \right) \left| \sum_{n \leq x} c(n)\lambda_f(n) \right|^2
 \nonumber \\
\ll_k 
D \sum_{i=1}^Y \sum_{j=1}^{i} \binom{i-1}{j-1} e^{-2jS} \nonumber \\
\ll_k D  e^{-2S}\sum_{i=1}^Y (1+e^{-2S})^{i-1} \nonumber \\
\ll_k D  Y e^{-S}. \label{eq:detcomparison}
\end{gather}
Arguing analogously, we also have
\begin{gather}
\mathbb{E} \prod_{i=1}^Y
h_{i} \left( Q_i(\mathbb{X}) \right) \left| \sum_{n \leq x} c(n)\mathbb{X}(n) \right|^2 
-\mathbb{E}\prod_{i=1}^Y
\widetilde{h_{i}} \left( Q_i(\mathbb{X}) \right) \left| \sum_{n \leq x} c(n)\mathbb{X}(n) \right|^2
 \nonumber\\
\ll D Y e^{-S} . \label{eq:randomcomparison}
\end{gather}
Therefore, combining (\ref{eq:taylormaindiff}), (\ref{eq:detcomparison}), and (\ref{eq:randomcomparison}), we conclude that
\begin{gather*}
 \mathbb{E}^h_{f \in \mathcal{H}_k(N)} \prod_{i=1}^Y
g_{j(i)} \left( S_i(f) \right) \left| \sum_{n \leq x} c(n)\lambda_f(n) \right|^2 
- \mathbb{E} \prod_{i=1}^Y
g_{j(i)} \left( S_i(\mathbb{X}) \right) 
\left| \sum_{n \leq x} c(n)\mathbb{X}(n) \right|^2  \\
= 
 \mathbb{E}^h_{f \in \mathcal{H}_k(N)} \prod_{i=1}^Y
h_{i} \left( Q_i(f) \right) \left| \sum_{n \leq x} c(n)\lambda_f(n) \right|^2 
- \mathbb{E} \prod_{i=1}^Y
h_{i} \left( Q_i(\mathbb{X}) \right) 
\left| \sum_{n \leq x} c(n)\mathbb{X}(n) \right|^2  \\
\ll_k D  \left(Y e^{-S} + \left( \frac{xP^{2SY}}{N} \right)^{k-1} Z^Y  \right).
\end{gather*}
Recall that
\begin{align*}
D=\sum_{n \leq x} d_{\mathcal{P}}(n),
\end{align*}
\begin{align*}
S=\lceil C_0 Y \delta^{-2} \log (J \log P) \rceil,
\end{align*}
and
\begin{align*}
Z=1+\frac{2J+1}{\delta} \cdot \sqrt{8S} \exp \left( \frac{2\pi^2 MY}{\delta^2} \right).
\end{align*}
Choosing $C_0$ to be sufficiently large
yields 
\begin{align*}
 Y e^{-S} \ll \frac{1}{(J\log P)^{2Y/\delta^2}}
\end{align*}
and $Z \leq \exp(S/100),$ so that
\begin{align*}
\left( \frac{xP^{2SY}}{N} \right)^{k-1} Z^Y \leq P^{-2(k-1)SY} \exp \left( \frac{SY}{100}\right)
\end{align*}
under the assumption $xP^{4SY}<N,$ which is negligible. In the case $k=2$, the corresponding contribution is bounded by
\begin{align*}
P^{-2SY}Z^Y\leq P^{-2SY}\exp(SY/100),
\end{align*}
which is also negligible. Finally, Mertens' estimate gives
\begin{align*}
D\leq x\prod_{p\leq P}\left(1-\frac{1}{p}\right)^{-1}\ll x\log P.
\end{align*}
Therefore, the proof is complete.
\end{proof}

\section{Proof of Theorem \ref{thm:modform}}


To prove Theorem \ref{thm:modform}, 
recall that if $f \in \mathcal{S}_k(\Gamma_0(N)),$ then the \textit{Fricke involution} of $f$ is defined as
\begin{align} \label{eq:fricke}
(f |_k W_N)(z):=N^{-k/2} z^{-k} f\left(-\frac{1}{Nz}\right) \qquad \text{for $z \in \mathbb{H}$}.
\end{align}
In fact, the Fricke operator is unitary, i.e.,
$\|f |_k W_N\|_{\rm Pet}=\|f\|_{\rm Pet}$
(see, e.g., \cite[Section 14.5]{MR2061214}). 

We split the proof of Theorem \ref{thm:modform} for the two ranges $1 \leq r \leq N$ and $r > N.$

\begin{proof}[Proof of Theorem \ref{thm:modform} for $1 \leq r \leq N$] 
We first prove the theorem for $1 \leq r \leq N_k.$
By Jensen's inequality, it suffices to consider $q \in [2/3,1],$ so that $2q-1>0.$ For $x \geq 1,$ denote
\begin{align*}
S_f(x):= \sum_{n \leq x} \lambda_f(n).
\end{align*}
By partial summation, we obtain
\begin{align*}
\sum_{n \leq N_k} \lambda_f(n) n^{\frac{k-1}{2}} e^{-\frac{2\pi n}{r}} 
= N_k^{\frac{k-1}{2}} e^{-\frac{2\pi N_k}{r}} S_f(N_k)
-\int_1^{N_k}  \left( \frac{k-1}{2x}-\frac{2\pi}{r} \right) x^{\frac{k-1}{2}} e^{-\frac{2\pi x}{r}} S_f(x) dx, 
\end{align*}
and therefore
\begin{gather} 
|f (i/r)| \leq {N_k}^{\frac{k-1}{2}} e^{-\frac{2\pi N_k}{r}} |S_f(N_k)|
+ \int_1^{N_k}  \left| \frac{k-1}{2x}-\frac{2\pi}{r} \right| x^{\frac{k-1}{2}} e^{-\frac{2\pi x}{r}} |S_f(x)| dx \nonumber \\
+ \left| \sum_{n>N_k}  \lambda_f(n) n^{\frac{k-1}{2}} e^{-\frac{2\pi n}{r}}  \right|.
\label{eq:afterpartial}
\end{gather}

To bound the contribution from the boundary term in (\ref{eq:afterpartial}), we apply Theorem \ref{thm:det} to get
\begin{align*}
\mathbb{E}_{f \in \mathcal{H}_k(N)}^h |S_f(N_k)|^{2q} \ll_k N_k^q,
\end{align*}
which implies that
\begin{align*}
 \mathbb{E}_{f \in \mathcal{H}_k(N)}^h ({N_k}^{\frac{k-1}{2}} e^{-\frac{2\pi N_k}{r}} |S_f(N_k)|)^{2q}
&\ll_k  N_{k}^{kq} \exp \left(-\frac{4\pi q N_k}{r} \right) \\
&= r^{kq} \left( \frac{N_k}{r} \right)^{kq} \exp \left(-\frac{4\pi q N_k}{r} \right).
\end{align*}
Set $R_k:=\min\left\{r,N_k/r\right\}.$ Since
\begin{align*} 
\log \log (10R_k) \ll_k \left( \frac{N_k}{r} \right)^{-2k}  \exp \left(\frac{8\pi  N_k}{r} \right),
\end{align*}
we conclude that
\begin{align} \label{eq:x=N_k}
 \mathbb{E}_{f \in \mathcal{H}_k(N)}^h ({N_k}^{\frac{k-1}{2}} e^{-\frac{2\pi N_k}{r}} |S_f(N_k)|)^{2q}
\ll_k   \left(\frac{r^{k}}{1+(1-q)\sqrt{\log \log (10R_k)}} \right)^q.
\end{align}

To bound the contribution from the integral in (\ref{eq:afterpartial}), we use the decomposition
\begin{align*}
[1,N_k]=[1,\sqrt{R_k}] \cup [\sqrt{R_k}, r\sqrt{R_k}] \cup [r\sqrt{R_k}, N_k] =:I_1 \cup I_2 \cup I_3,
\end{align*}
and denote
\begin{align*}
J_i(f):= \int_{I_i}  \left| \frac{k-1}{2x}-\frac{2\pi}{r} \right| x^{\frac{k-1}{2}} e^{-\frac{2\pi x}{r}} |S_f(x)| dx
\end{align*}
for $i=1,2,3.$ Since $2q-1>0,$ H\"older's inequality gives
\begin{gather} 
\mathbb{E}_{f \in \mathcal{H}_k(N)}^h  |J_i(f)|^{2q} \leq \left( \int_{I_i} \left| \frac{k-1}{2x}-\frac{2\pi}{r} \right| x^{\frac{k-1}{2}} e^{-\frac{2\pi x}{r}} dx \right)^{2q-1}  \nonumber\\
\cdot \int_{I_i} \left| \frac{k-1}{2x}-\frac{2\pi}{r} \right| x^{\frac{k-1}{2}} e^{-\frac{2\pi x}{r}} \cdot \mathbb{E}_{f \in \mathcal{H}_k(N)}^h  |S_f(x)|^{2q} dx.
\label{eq:jiholder}
\end{gather}

For the integral $J_1,$ Theorem \ref{thm:det} gives
\begin{align*}
\mathbb{E}_{f \in \mathcal{H}_k(N)}^h |S_f(x)|^{2q} \ll_k x^q.
\end{align*}
Combining this with
\begin{align*}
\int_1^{\sqrt{R_k}}
\left|\frac{k-1}{2x}-\frac{2\pi}{r}\right|
x^{\frac{k-1}{2}}e^{-\frac{2\pi x}{r}}\,dx
\ll_k R_k^{\frac{k-1}{4}},
\end{align*}
it follows from (\ref{eq:jiholder}) that
\begin{align*}
\mathbb{E}_{f \in \mathcal{H}_k(N)}^h |J_1(f)|^{2q}
&\ll_k R_k^{\frac{k-1}{4}\cdot (2q-1)} R_k^{\frac{k-1}{4}+\frac{q}{2}} \\
&= R_k^{kq/2}.
\end{align*}
Since $\log \log (10R_k) \ll R_k^{k},$ we have
\begin{align} \label{eq:j1}
\mathbb{E}_{f \in \mathcal{H}_k(N)}^h |J_1(f)|^{2q} \ll_k 
\left( \frac{r^k}{1+(1-q)\sqrt{\log \log (10R_k)}} \right)^q.
\end{align}

For the integral $J_2,$ note that $\min\{ x, N_k/x  \} \geq \sqrt{R_k}$ for $x \in I_2,$
and Theorem \ref{thm:det} gives 
\begin{align} \label{eq:j2thm}
\mathbb{E}_{f \in \mathcal{H}_k(N)}^h |S_f(x)|^{2q} \ll_k
\left( \frac{x}{1+(1-q)\sqrt{\log \log (10R_k)}} \right)^q.
\end{align}
Let $\alpha \in [0,1].$ Making the change of variables $y=x/r,$ we have
\begin{align} \label{eq:j2crude}
\int_{0}^{\infty} 
\left| \frac{k-1}{2x}-\frac{2\pi}{r} \right| x^{\frac{k-1}{2}} e^{-\frac{2\pi x}{r}} x^{\alpha} dx
&= r^{\frac{k-1}{2}+\alpha} \int_0^{\infty} \left| \frac{k-1}{2y}-2\pi \right| y^{\frac{k-1}{2}+\alpha}
e^{-2\pi y} dy \nonumber\\
&\ll_k r^{\frac{k-1}{2}+\alpha}. 
\end{align}
Taking $\alpha \in \{0,q\},$ it follows from (\ref{eq:jiholder}), (\ref{eq:j2thm}), and (\ref{eq:j2crude}) that
\begin{align} \label{eq:j2}
\mathbb{E}_{f \in \mathcal{H}_k(N)}^h |J_2(f)|^{2q}
&\ll_k r^{\frac{k-1}{2} \cdot (2q-1)}  \cdot
\frac{r^{\frac{k-1}{2}+q}}{(1+(1-q)\sqrt{\log \log (10R_k)})^{q}} \nonumber\\
&=  \left(\frac{r^{k}}{1+(1-q)\sqrt{\log \log (10R_k)}} \right)^q.
\end{align}

For the integral $J_3,$ let $\alpha \in [0,1].$ Making the change of variables $y=x/r,$
we have
\begin{align} \label{eq:j3crude}
\int_{r\sqrt{R_k}}^{\infty}  \left| \frac{k-1}{2x}-\frac{2\pi}{r} \right| x^{\frac{k-1}{2}} e^{-\frac{2\pi x}{r}} x^{\alpha} dx
&= r^{\frac{k-1}{2}+\alpha} \int_{\sqrt{R_k}}^{\infty} \left| \frac{k-1}{2y}-2\pi \right| y^{\frac{k-1}{2}+\alpha}
e^{-2\pi y} dy \nonumber \\
&\ll_k r^{\frac{k-1}{2}+\alpha}  e^{-\pi \sqrt{R_k}}. 
\end{align}
Theorem \ref{thm:det} yields
\begin{align*}
\mathbb{E}_{f \in \mathcal{H}_k(N)}^h |S_f(x)|^{2q} \ll_k x^q,
\end{align*}
and combining this with \eqref{eq:jiholder} and \eqref{eq:j3crude} (with $\alpha \in \{0,q\}$) that
\begin{align*}
\mathbb{E}_{f \in \mathcal{H}_k(N)}^h |J_3(f)|^{2q}
&\ll_k (r^{\frac{k-1}{2}}  e^{-\pi \sqrt{R_k}})^{2q-1} \cdot r^{\frac{k-1}{2}+q} e^{-\pi \sqrt{R_k}} \\
&= r^{kq} e^{-2\pi q \sqrt{R_k}}.
\end{align*}
Since $\log \log (10R_k) \ll  e^{4\pi  \sqrt{R_k}},$ we conclude that
\begin{align} \label{eq:j3}
\mathbb{E}_{f \in \mathcal{H}_k(N)}^h |J_3(f)|^{2q} \ll_k 
 \left(\frac{r^{k}}{1+(1-q)\sqrt{\log \log (10R_k)}} \right)^q.
\end{align}

To bound the contribution from the sum in (\ref{eq:afterpartial}), denote
$M_j:=2^j N_k$ for $j \geq 0.$ By Lemma \ref{lem:largesieveineq2}, we have
\begin{align*}
 \mathbb{E}_{f \in \mathcal{H}_k(N)}^h \left| \sum_{M_j < n \leq 2M_j} 
 \lambda_f(n) n^{\frac{k-1}{2}} e^{-\frac{2\pi n}{r}} 
 \right|^{2} &\ll_k \left( 1+\frac{M_j}{N}\right) 
 \sum_{M_j < n \leq 2M_j} 
 n^{k-1} e^{-\frac{4\pi n}{r}} \\
 &\ll_k \left( 1+\frac{M_j}{N}\right) M_j^{k} \exp \left( -\frac{4\pi M_j}{r} \right) \\
 &\ll_k 2^j M_j^{k}\exp \left( -\frac{4\pi M_j}{r} \right).
\end{align*}
Therefore, Minkowski's inequality gives
\begin{align*}
\left(\mathbb{E}_{f \in \mathcal{H}_k(N)}^h \left| \sum_{n>N_k}  \lambda_f(n) n^{\frac{k-1}{2}} e^{-\frac{2\pi n}{r}}  \right|^{2} \right)^{1/2}
&\leq \sum_{j \geq 0} \left(\mathbb{E}_{f \in \mathcal{H}_k(N)}^h \left| \sum_{M_j < n \leq 2M_j}  \lambda_f(n) n^{\frac{k-1}{2}} e^{-\frac{2\pi n}{r}}  \right|^{2} \right)^{1/2} \\
&\ll_k \sum_{j \geq 0} 2^{j/2} M_j^{k/2}\exp \left( -\frac{2\pi M_j}{r} \right) \\
&\ll_k N_k^{k/2}\sum_{j \geq 0} 2^{j(k+1)/2}  \exp \left( -\frac{2^{j+1}\pi N_k}{r} \right).
\end{align*}
Since by assumption $r \leq N_k,$ this is
\begin{align*}
\ll_k N_k^{k/2} \exp \left( -\frac{\pi N_k}{r} \right),
\end{align*}
and hence
\begin{align*}
\mathbb{E}_{f \in \mathcal{H}_k(N)}^h \left| \sum_{n>N_k}  \lambda_f(n) n^{\frac{k-1}{2}} e^{-\frac{2\pi n}{r}}  \right|^{2} &\ll_k N_k^{k} \exp \left( -\frac{2\pi N_k}{r} \right) \\
&= r^k \left( \frac{N_k}{r} \right)^{k} \exp \left( -\frac{2\pi N_k}{r} \right).
\end{align*}
Since 
\begin{align*}
\log \log (10R_k) \ll_k \left( \frac{N_k}{r} \right)^{-2k}  \exp \left(\frac{4\pi  N_k}{r} \right),
\end{align*}
we obtain
\begin{align*}
\mathbb{E}_{f \in \mathcal{H}_k(N)}^h \left| \sum_{n>N_k}  \lambda_f(n) n^{\frac{k-1}{2}} e^{-\frac{2\pi n}{r}}  \right|^{2} \ll_k \frac{r^k}{1+(1-q)\sqrt{\log \log (10R_k)}}.
\end{align*}
By Jensen's inequality, we conclude that
\begin{align} \label{eq:tail}
\mathbb{E}_{f \in \mathcal{H}_k(N)}^h \left| \sum_{n>N_k}  \lambda_f(n) n^{\frac{k-1}{2}} e^{-\frac{2\pi n}{r}}  \right|^{2q} &\leq
\left(\mathbb{E}_{f \in \mathcal{H}_k(N)}^h \left| \sum_{n>N_k}  \lambda_f(n) n^{\frac{k-1}{2}} e^{-\frac{2\pi n}{r}}  \right|^{2} \right)^q
\nonumber \\
&\ll_k \left(\frac{r^{k}}{1+(1-q)\sqrt{\log \log (10R_k)}} \right)^q.
\end{align}
Combining (\ref{eq:x=N_k}), (\ref{eq:j1}), (\ref{eq:j2}), (\ref{eq:j3}), and (\ref{eq:tail}), the bound
\begin{align} \label{eq:rleqNk}
\mathbb{E}_{f \in \mathcal{H}_k(N)}^h |f(i/r)|^{2q} \ll_{k} 
\left( \dfrac{r^k}{1+(1-q)\sqrt{\log \log (10R_k)}} \right)^q
\end{align}
follows for $1 \leq r \leq N_k.$ 

Since $N_k=N$ for $k \geq 4,$ it remains to handle the case $k=2.$
We first suppose that $1\leq r\leq \sqrt{N}.$ Since $\sqrt{N}\leq N_2$
for sufficiently large $N,$ the bound (\ref{eq:rleqNk}) applies. Moreover,
if $r\leq \sqrt{N_2},$ then $R_2=r,$ whereas if
$\sqrt{N_2}<r\leq \sqrt{N},$ then
\begin{align*}
R_2=\frac{N_2}{r}\geq \frac{\sqrt{N}}{\log N}.
\end{align*}
Therefore, we have
\begin{align*}
1+(1-q)\sqrt{\log\log(10R_2)}
\gg
1+(1-q)\sqrt{\log\log(10r)},
\end{align*}
and (\ref{eq:rleqNk}) gives
\begin{align} \label{eq:rleqsqrtn}
\mathbb{E}_{f \in \mathcal{H}_2(N)}^h |f(i/r)|^{2q}
\ll
\left(
\frac{r^2}{1+(1-q)\sqrt{\log\log(10r)}}
\right)^q
\end{align}
for $1\leq r\leq \sqrt{N}.$

Now suppose that $\sqrt{N}<r\leq N$ and put $t:=N/r,$ so that
$1\leq t<\sqrt{N}.$ Since every
$f\in\mathcal{H}_2(N)$ is primitive, we have
$f|_2W_N=\epsilon_f f$ for some $|\epsilon_f|=1,$ and therefore
\begin{align*}
\left|f\left(\frac{i}{r}\right)\right|
=
\frac{r^2}{N}
\left|f\left(\frac{i}{t}\right)\right|.
\end{align*}
Using (\ref{eq:rleqsqrtn}) at $t=N/r,$ we obtain
\begin{align*}
\mathbb{E}_{f \in \mathcal{H}_2(N)}^h
\left|f\left(\frac{i}{r}\right)\right|^{2q}
&=
\left(\frac{r^2}{N}\right)^{2q}
\mathbb{E}_{f \in \mathcal{H}_2(N)}^h
\left|f\left(\frac{i}{t}\right)\right|^{2q}\\
&\ll
\left(\frac{r^2}{N}\right)^{2q}
\left(
\frac{t^2}{1+(1-q)\sqrt{\log\log(10t)}}
\right)^q\\
&=
\left(
\frac{r^2}{1+(1-q)\sqrt{\log\log(10N/r)}}
\right)^q.
\end{align*}
Since $R=N/r$ in this range, the theorem follows for $1\leq r\leq N.$
\end{proof}

\begin{proof}[Proof of Theorem \ref{thm:modform} for $r>N$]
Let $f \in \mathcal{H}_k(N).$ By definition, we have
\begin{align*}
\left|f \left( \frac{i}{r} \right) \right|
=
\left(\frac{r^2}{N} \right)^{k/2}
\left|(f |_k W_N)\left( \frac{ir}{N} \right) \right|,
\end{align*}
which gives
\begin{align} \label{eq:afterfricke}
\mathbb{E}_{f \in \mathcal{H}_k(N)}^h
\left|f \left( \frac{i}{r} \right) \right|^{2}
=
\left( \frac{r^2}{N} \right)^{k}
\mathbb{E}_{f \in \mathcal{H}_k(N)}^h
\left| (f |_k W_N)\left( \frac{ir}{N} \right) \right|^{2}.
\end{align}
Denote
\begin{align*}
\widehat{\mathcal{H}_k(N)}
:=
\{f|_kW_N:f\in\mathcal{H}_k(N)\}.
\end{align*}
Since the Fricke involution is unitary, the set 
$\widehat{\mathcal{H}_k(N)}$ is also an orthogonal basis of
$\mathcal{S}_k(\Gamma_0(N)),$ and hence
\begin{align} \label{eq:unitary}
\mathbb{E}_{f \in \mathcal{H}_k(N)}^h
\left| (f |_k W_N)\left( \frac{ir}{N} \right) \right|^{2}
=
\mathbb{E}_{g \in \widehat{\mathcal{H}_k(N)}}^h
\left| g\left( \frac{ir}{N} \right) \right|^{2}.
\end{align}

Let $g \in \widehat{\mathcal{H}_k(N)}.$ We write
\begin{align*}
g\left(  \frac{ir}{N} \right)
&=
\lambda_g(1)e^{-\frac{2\pi r}{N}}
+
\sum_{j \geq 0}
\sum_{2^j<n \leq 2^{j+1}}
\lambda_g(n)n^{\frac{k-1}{2}}
e^{-\frac{2\pi n r}{N}}.
\end{align*}
By Lemma \ref{lem:largesieve} with $x=1,$ we have
\begin{align*}
\mathbb{E}_{g \in \widehat{\mathcal{H}_k(N)}}^h
|\lambda_g(1)|^2
\ll_k 1.
\end{align*}
Also, applying Lemma \ref{lem:largesieveineq2} for each $j \geq 0,$
we obtain
\begin{align*}
\mathbb{E}_{g \in \widehat{\mathcal{H}_k(N)}}^h
\left|
\sum_{2^j<n \leq 2^{j+1}}
\lambda_g(n)n^{\frac{k-1}{2}}
e^{-\frac{2\pi n r}{N}}
\right|^2
&\ll_k
\left(1+\frac{2^j}{N}\right)
\sum_{2^j<n \leq 2^{j+1}}
n^{k-1}e^{-\frac{4\pi nr}{N}}\\
&\ll_k
2^{j(k+1)}
\exp\left(-2^j\cdot\frac{4\pi r}{N}\right).
\end{align*}
Therefore, Minkowski's inequality gives
\begin{align*}
\left(
\mathbb{E}_{g \in \widehat{\mathcal{H}_k(N)}}^h
\left|g\left(\frac{ir}{N}\right)\right|^2
\right)^{1/2}
&\ll_k
e^{-\frac{2\pi r}{N}}
+
\sum_{j\geq0}
2^{\frac{j(k+1)}{2}}
\exp\left(-2^j\cdot\frac{2\pi r}{N}\right).
\end{align*}
Since $r>N,$ this is
$\ll_k \exp\left(-2\pi r/N\right),$
and hence
\begin{align*}
\mathbb{E}_{g \in \widehat{\mathcal{H}_k(N)}}^h
\left|g\left(\frac{ir}{N}\right)\right|^2
\ll_k
\exp\left(-\frac{4\pi r}{N}\right).
\end{align*}
Combining this with (\ref{eq:afterfricke}) and (\ref{eq:unitary}), we obtain
\begin{align*}
\mathbb{E}_{f \in \mathcal{H}_k(N)}^h
|f(i/r)|^2
\ll_k
\left(\frac{r^2}{N}\right)^k
\exp\left(-\frac{4\pi r}{N}\right).
\end{align*}
Finally, Jensen's inequality gives
\begin{align*}
\mathbb{E}_{f \in \mathcal{H}_k(N)}^h
|f(i/r)|^{2q}
&\leq
\left(
\mathbb{E}_{f \in \mathcal{H}_k(N)}^h
|f(i/r)|^2
\right)^q\\
&\ll_k
\left(
\left(\frac{r^2}{N}\right)^k
\exp\left(-\frac{4\pi r}{N}\right)
\right)^q.
\end{align*}
Therefore, the theorem follows whenever $r>N$ as well, and the proof is complete.
\end{proof}

\printbibliography


\end{document}